\documentclass{birkart}
\usepackage{amsmath,amssymb,amsthm,amscd}
\def\nt{\noindent}
\newtheorem{thm}{Theorem}[section]
\newtheorem{cor}[thm]{Corollary}
\newtheorem{lem}[thm]{Lemma}
\newtheorem{prop}[thm]{Proposition}
\theoremstyle{definition}
\newtheorem{defn}[thm]{Definition}
\theoremstyle{remark}
\newtheorem{rem}[thm]{Remark}
\newtheorem{ex}{Example}
\numberwithin{equation}{section}
\newcommand{\R}{\mathbb{R}}

\newcommand{\N}{\mathbb{N}}
\newcommand{\M}{\mathbb{M}}
\newcommand{\Z}{\mathbb{Z}}
\newcommand{\C}{\mathbb{C}}

\newcommand{\Ocal}{\mathcal{O}}
\newcommand{\Kcal}{\mathcal{K}}
\newcommand{\Acal}{\mathcal{A}}
\newcommand{\Scal}{\mathcal{S}}
\newcommand{\Ecal}{\mathcal{E}}
\newcommand{\Lcal}{\mathcal{L}}
\newcommand{\Hcal}{\mathcal{H}}

\newcommand{\Wcal}{\mathcal{W}}

\newcommand{\dbar}{{d\hspace{-0,05cm}\bar{}\hspace{0,05cm}}}
\newcommand{\Op}{\textup{Op}}

\newcommand{\op}{\textup{op}}

\newcommand{\td}{\tilde}

\newcommand{\beq}{\begin{equation}}
\newcommand{\eeq}{\end{equation}}
\newcommand{\virg}[1]{`{#1}'}

\begin{document}
%
\title[Asymptotic Parametrices of Elliptic Edge Operators]
{Asymptotic Parametrices of Elliptic Edge Operators}

\author{H.-J. Flad}
\address{Institute of Mathematics \\
Technical University Berlin\\
Strasse des 17. Juni 136\\
D-10623 Berlin\\
Germany}
\email{flad@mail.math.tu-berlin.de}

\author{G. Harutyunyan}
\address{Institute of Mathematics\\
Carl von Ossietzky University\\
Carl-von-Ossietzky-Str. 9-11\\
D-26129 Oldenburg\\
Germany}
\email{harutyunyan@mathematik.uni-oldenburg.de}

\author[B.-W. Schulze]{B.-W. Schulze}
\address{Institute of Mathematics, University of Potsdam, Am Neuen Palais 10\\
D-14469 Potsdam\\
Germany}
\email{schulze@math.uni-potsdam.de}

\subjclass{Primary 35S35; Secondary 35J70}

\keywords{cone and edge pseudo-differential operators, ellipticity of edge-degenerate operators, \\ 
meromorphic operator-valued symbols, asymptotics of solutions}

\date{}
  
\today


\begin{abstract}
We study operators on a singular manifold, here of conical or edge type, and develop a new general approach of representing asymptotics of solutions to elliptic equations close to the singularities. The idea is to construct so-called asymptotic parametrices with flat left-over terms. Our structures are motivated by models of particle physics with singular potentials that contribute embedded singularities in $\R^N$ of higher order, according to the number of particles.  

\end{abstract}

\maketitle
\tableofcontents
%
%
%
%
\section*{Introduction}
Ellipticity of an operator $A$ on a singular configuration $M,$ here with a stratification $s(M):=(s_j(M))_{j=0,\dots,k}$, cf. the definition below, relies on a principal symbolic structure $\sigma (A):=(\sigma _j(A))_{j=0,\dots,k},$ more precisely, on a componentwise invertibility (or Fredholm) condition, cf. \cite{Schu57}. In the simplest case when $M$ is smooth (which corresponds to $k=0$) and if $A$ is a classical pseudo-differential operator of order $\mu \in \R,$ e.g., a differential operator with smooth coefficients,  then $\sigma _0(A)$ is the standard homogeneous principal symbol of order $\mu .$ The case $k=1$ corresponds to conical or edge singularities. Those form a subset $s_1(M)\subset M$, and we set $s_0(M):=M\setminus s_1(M);$ both $s_0(M)$ and $s_1(M)$ are smooth manifolds. Now $\sigma _0(A)$ is the principal symbol of the operator $A$ on $s_0(M),$ while $\sigma _1(A)$ is an operator-valued symbol associated with $s_1(M).$ \\

\nt In the case of conical singularities, i.e. $\textup{dim}\, s_1(M)=0,$ without loss of generality we assume that  $s_1(M)$ consists of a single point $v$. Then $M$ is locally near $v$ modelled on a cone
\beq\label{int.cone}
X^\Delta :=(\overline{\R}_+\times X)/(\{0\}\times X)
\eeq
for a closed compact $C^\infty $ manifold $X.$ Let $\textup{Diff}^\mu (\cdot)$ be the space of differential operators of order $\mu $ on the smooth manifold in parentheses, with smooth coefficients (in its natural Fr\'echet topology). Operators on $M$ close to $v$ are expressed in the splitting of variables $(r,x)\in \R_+\times X=:X^\wedge.$ The space $\textup{Diff}_{\textup{deg}}^\mu (M)$ is then defined to be the set of all $A\in \textup{Diff}^\mu (M\setminus \{v\})$ that are locally near $v$ in the splitting of variables $(r,x)$ of the form  
\beq\label{i.degc}
A=r^{-\mu }\sum_{j=0}^\mu a_j(r)(-r\partial _r)^j
\eeq
for coefficients $a_j\in C^\infty (\overline{\R}_+,\textup{Diff}^{\mu -j}(X)).$ The first component of the pair $\sigma (A)=(\sigma _0(A),\sigma _1(A))$ of principal symbols is the standard homogeneous principal symbol (the interior symbol) of $A$ as an operator in $\textup{Diff}^\mu (M\setminus \{v\}),$ and $\sigma _1(A),$ the conormal symbol of $A,$ is defined as the $z$-dependent operator family
\beq\label{int.conc}
\sigma _1(A))(z)=\sum_{j=0}^\mu a_j(0)z^j:H^s(X)\rightarrow H^{s-\mu }(X)
\eeq
between standard Sobolev spaces on $X.$ The ellipticity of $A$ with respect to $\sigma (A)=(\sigma _0(A),\sigma _1(A))$ for a fixed weight $\gamma \in \R$ is defined as follows: The interior symbol $\sigma _0(A)$ never vanishes on $T^*(s_0(M))\setminus 0,$ and locally near $v$ in the variables $(r,x)\in \R_+\times X$ with the covariables $(\varrho ,\xi )$ we have 
\beq\label{i.E.red}
\tilde{\sigma }_0(A)(r,x,\varrho ,\xi ):=r^\mu \sigma _0(A)(r,x,r^{-1}\varrho ,\xi )\neq 0
\eeq
for $(\varrho ,\xi)\neq 0,$ up to $r=0;$ moreover, \eqref{int.conc} is bijective for all $z\in \Gamma _{(n+1)/2-\gamma }$ and any $s\in \R.$ Here $n:=\textup{dim}\,X,$ and
\beq\label{int.cwei}
\Gamma _\beta :=\{z\in \C:\textup{Re}\,z=\beta \}. 
\eeq
The function $\tilde{\sigma }_0(A)$ is also referred to as the rescaled interior symbol of $A.$
As is well-known if $A$ is elliptic with respect to $\sigma _0$ the space $L^{-\mu }_{\textup{cl}}(s_0(M))$ of classical pseudo-differential operators of order $-\mu $
contains elements $P$ such that $G:=PA-1$ and $C:=AP-1$ have kernels in $C^\infty (s_0(M)\times s_0(M)).$ Moreover, under the ellipticity conditions with respect to 
$(\sigma _0(A),\sigma _1(A))$ we find $P$ in the cone algebra, cf. \cite{Schu2} and Section 1.2 below. In this case the left-over terms are so-called Green operators with discrete asymptotics. From $P$ we may derive asymptotics of solutions, cf. notation in Section 1.2. We also obtain regularity in weighted Sobolev spaces over $s_0(M).$ In this article we construct asymptotic parametrices; those deliver the asymptotic information by a successive process, cf. Section 1.3. The advantage is that the computation of asymptotics becomes more transparent and manageable, cf. \cite{Flad3}. Analogous constructions are useful for edge singularities and singularities of higher order. The edge case will be studied in Section 2. Higher order singularities occur in many-particle systems. Corresponding asymptotic parametrices are interesting as well; the construction is planned in a forthcoming paper. In any case, the techniques for corner and edge singularities are crucial for the iterative concept for treating higher order singularities. It turns out that it is essential to interpret the (degenerate) operators in question as elements of operator algebras, in particular, to consider compositions. Then it is necessary to admit from the very beginning also higher order operators (and of negative order when we express parametrices of differential operators), even if we are finally interested in applications in some specific equation, say, of second order. Some elements of the pseudo-differential machinery will be formulated below, but the nature of edge degeneration and of the principal symbolic structure can be illustrated for differential operators.
A manifold $M$ with smooth edge $s_1(M)$ of dimension $q>0$ is locally near any point of the edge modelled on 
\beq\label{modeg}
X^\Delta \times \Omega
\eeq 
where $\Omega \subseteq \R^q$ corresponds to a chart on $s_1(M)$ with local coordinates $y,$ and $X$ is a closed compact $C^\infty $ manifold.
Then $\textup{Diff}^\mu _{\textup{deg}}(M)$ for a manifold $M$ with smooth edge $Y$ denotes the space of those $A\in \textup{Diff}^\mu (s_0(M))$ that are locally close to the edge in the splitting of variables $(r,x,y)\in X^\wedge\times \Omega $ of the form
\beq\label{iedeg}
A=r^{-\mu }\sum_{j+|\alpha  |\leq \mu } a_{j\alpha }(r,y)(-r\partial _r)^j(rD_y)^\alpha 
\eeq
for coefficients $a_{j\alpha }\in C^\infty (\overline{\R}_+\times \Omega ,\textup{Diff}^{\mu -(j+|\alpha |)}(X)),$ see the article \cite{Schu32}, or the monographs \cite{Schu2}, \cite{Egor1}, \cite{Schu20}. Parallel to the stratification $M=s_0(M)\cup s_1(M)$ as a disjoint union of smooth manifolds 
the operators $\textup{Diff}^\mu _{\textup{deg}}(M)$ have a principal symbolic structure $\sigma (A)=(\sigma _0(A),\sigma _1(A))$ where $\sigma _0(A)\in C^\infty (T^*(s_0(M))\setminus 0)$ is the standard homogeneous principal symbol of $A$ as an operator of $\textup{Diff}^\mu (s_0(M)),$ while 
\beq\label{isyeg}
\sigma _1(A)(y,\eta ):=r^{-\mu }\sum_{j+|\alpha  |\leq \mu } a_{j\alpha }(0,y)(-r\partial _r)^j\newline (r\eta )^\alpha :\Kcal^{s,\gamma }(X^\wedge)\rightarrow \Kcal^{s-\mu ,\gamma -\mu }(X^\wedge)
\eeq
 is operator-valued, acting as a family of operators on the open stretched cone $X^\wedge,$ parametrised by $(y,\eta )\in T^*(s_1(M))\setminus  0$ (the cotangent bundle of $s_1(M)$ minus the zero section). The definition of the weighted ``Kegel"-spaces $\Kcal^{s,\gamma }(X^\wedge)$ will be repeated below, cf. Section 2.4. The ellipticity of $A$ with respect to $\sigma (A)=(\sigma _0(A),\sigma _1(A))$ for a fixed weight $\gamma \in \R$ is defined as follows: The interior symbol $\sigma _0(A)$ never vanishes on $T^*( s_0(M))\setminus 0,$ and locally near $s_1(M)$ in the variables $(r,x,y)\in \R_+\times X\times \Omega $ with the covariables $(\varrho ,\xi ,\eta )$ the rescaled symbol 
\beq\label{i.E.redy}
\tilde{\sigma }_0(A)(r,x,y,\varrho ,\xi ,\eta ):=r^\mu \sigma _0(A)(r,x,y,r^{-1}\varrho ,\xi ,r^{-1}\eta )
\eeq
does not vanish for $(\varrho ,\xi,\eta )\neq 0,$ up to $r=0;$ moreover, \eqref{isyeg} is bijective for all $(y,\eta )\in T^*Y\setminus 0$ and any $s\in \R.$ It is well-known that the ellipticity condition with respect to $\sigma _0(A)$  entails the Fredholm property of \eqref{isyeg} for every $\gamma $ off some discrete set $D(y)$ of reals. In this case, when a certain topological obstruction vanishes (which is the case in our applications) we can enlarge \eqref{isyeg} by finite rank operators to a bijective $2\times 2$ block matrix family with $\sigma _1(A)(y,\eta )$ in the upper left corner. These extra entries encode trace and potential operators on the level of edge symbols, to guarantee the existence of local parametrices, cf. also Remark \ref{22.R.ellcon} below. For simplicity we assume for the moment that \eqref{isyeg} itself is bijective, without such additional data. It may be difficult to explicitly decide for which weights $\gamma $ this is the case (by virtue of relative index results under changing $\gamma $ the index of \eqref{isyeg} can be modified). Concrete results on the index of elliptic edge symbols may be found in \cite{Haru13}. Bijectivity results are also announced in \cite{Flad4}. Another useful observation is that there are always so-called smoothing Mellin plus Green symbols $\sigma _1(M+G)(y,\eta )$ of the edge calculus such that $\sigma _1(A+M+G)(y,\eta )$ is bijective. Below in our constructions we will say more about such structures; let us only note at this point that the bijectivity does not depend on $s$ but depends on $(y,\eta /|\eta |).$ In our construction of asymptotic parametrices  under such an ellipticity condition we will refer to the above-mentioned discrete exceptional weights that are determined by the non-bijectivity points of the subordinate conormal symbol
\beq\label{coyeg}
\sigma _1(\sigma _1(A))(y,z):=\sum_{j=0}^ \mu  a_{j0 }(0,y)z^j:H^s(X)\rightarrow H^{s-\mu }(X)
\eeq
pointwise computed for $\sigma _1(A))(y,\eta )\in C^\infty (T^*\Omega \setminus 0,\textup{Diff}_{\textup{deg}}^\mu (X^\wedge)),$ according to \eqref{int.conc}. Later on, in order to avoid confusion with the different $\sigma _1$ we also set $\sigma _{\textup{c}}(A)(y,z):=\sigma _1(\sigma _1(A))(y,z).$\\
\nt The main intention of this paper is to develop a new concise approach to pseudo-differential cone and edge algebras and to elliptic regularity with asymptotics. The recent experience of the authors with the problem of explictly computing the asymptotic data of solutions, cf. the articles \cite{Flad3}, \cite{Flad4}, is that it needs parallel work to organise the construction of parametrices and of the underlying symbolic structures in a more transparent and accessible way. Of course, we refer to the basics of the cone and edge calculus, known from a number of articles and monographs, cf. \cite{Schu32}, \cite{Schu2}, \cite{Egor1}, \cite{Schu20}. Already several typical applications to the crack theory or to mixed problems required once again voluminous extensions, cf. \cite{Kapa10} or \cite{Haru13}. The concrete models that we have in mind here and in future should be understandable in terms of a minimal theoretical machinery. The point is that in particle physics relatively simple elliptic operators (essentially Laplacians on a smooth manifold, e.g. in $\R^{3N}$) are combined with singular potentials of the type $\sum_{j,k\leq N}c_{jk}|{\bf{x^j}}-{\bf{x^k}}|^{-1}$ where ${\bf{x^j}}=(x^j_1,x^j_2,x^j_3)$ represents the position of the $j$-th particle in $\R^3.$ The singularities of the potential give rise to a corner manifold, embedded in $\R^{3N},$ and the problem is to understand the asymptotics of solutions close to that singular set. This is a task of ``unexpected" complexity, even for small $N \,(=1 \,\,\mbox{or}\,\, 2),$ corresponding to conical or edge singularities. We found it therefore justified to reorganise the way of deriving asymptotics of solutions by means of so-called asymptotic parametrices which are introduced in this paper both for the conical and the edge case. Since the concept is based on ``ordinary" parametrix constructions we also discuss this aspect here, either with new proofs or corresponding references. \\
\nt This paper is organised as follows. In Section 1 we formulate the cone algebra including its symbolic structure, ellipticity and (asymptotic) parametrices. This material may also be regarded as a tool and simplest model for the case of edge singularities in Section 2 which is the main part of this article. We establish the edge algebra with discrete asymptotics, including the symbolic machinery and the structure of ellipticity with (asymptotic) parametrices.
In addition we discuss several special cases and simplifications of the general approach, e.g. when (as in our applications to the above-mentioned case, to be published elsewhere) the original operator is a differential operator. Let us finally note that larger $N$ requires the calculus for higher order corners. Also this technique should be further smoothened and simplified, to make the case of higher particle numbers accessible. 
\section{Elliptic cone operators} 
\subsection{Operators of Fuchs type} 
Let $M$ be a manifold with conical singularity $v:=s_1(M).$ Because of \eqref{int.cone} we have a local identification of $M\setminus \{v\}$ with $X^\wedge$ which allows us to attach a copy of $X$ to $M\setminus \{v\}.$ In this way we get the stretched manifold $\M$ associated with $M$ which is a smooth manifold with boundary $\partial \M\cong  X.$ A collar neighbourhood of $\partial \M$ may be identified with $\overline{\R}_+\times X$ (or, equivalently, with $[0,1)\times X $). For purposes below on $M\setminus \{v\}$ we fix a strictly positive function ${\bf{r}}^1\in C^\infty (M\setminus \{v\})$ that is equal to $r$ close to $v$ in a fixed splitting of variables $(r,x)\in \R_+\times X.$ For any $\beta \in\R$ we call ${\bf{r}}^\beta $ a weight function with weight $\beta .$\\

\nt The cone algebra on $M$ is defined as a union of subspaces $L^\mu (M,{\bf{g}})\subset L_{\textup{cl}}^\mu(M\setminus \{v\}),\mu \in \R,$ labelled with weight data ${\bf{g}}:=(\gamma ,\gamma -\mu ,\Theta )$ for a weight $\gamma \in \R$ and a weight interval $\Theta ,$ to be specified below. Here $L_{\textup{cl}}^\mu(\cdot)$  is the space of classical pseudo-differential operators on the manifold in parentheses (assumed to be Riemannian, with an asociated metric $dx$). We also employ the parameter-dependent variant $L_{\textup{cl}}^\mu(\cdot\,;\R^l),\lambda \in \R^l,$ i.e. the space of families of operators $A(\lambda )$ in  $L_{\textup{cl}}^\mu(\cdot),$ modulo $L^{-\infty }(\cdot\,;\R^l):= \Scal(\R^l,L^{-\infty }(\cdot))$ described by local amplitude functions $a(x,\xi ,\lambda )$ with $(\xi ,\lambda )$ being treated as covariables.
One of the ingredients of the cone algebra is the space $M^\mu_{\mathcal{O}}(X)$ of holomorphic $L_{\textup{cl}}^\mu (X)$-valued Mellin symbols. The background is that the operator $-r\partial _r$ occurring in \eqref{i.degc} can be expressed as a Mellin pseudo-differential operator on the half-axis with the symbol $z$, namely, $-r\partial _r=M^{-1}zM, $ say, as an operator $C_0^\infty (\R_+)\rightarrow C_0^\infty (\R_+).$ Here
\beq\label{11.E.mel}
Mu(z)=\int_0^\infty r^zu(r)dr/r\,\,\,\mbox{with the inverse}\,\,\,M^{-1}g(r)=\int_{\Gamma _{1/2-\gamma }}r^{-z}g(z)\dbar z
\eeq
for any $\gamma \in \R,\dbar z:= (2\pi i)^{-1}dz.$\\

\nt In the edge calculus below we need Mellin symbols also in parameter-dependent form. Let us recall the definition.
First, if $U\subseteq \C$ is an open set and $E$ a Fr\'echet space, by $\Acal(U,E)$ we denote the space of holomorphic functions in $U$ with values in $E$ in the Fr\'echet topology of uniform convergence on compact sets. The space $M^\mu _{\mathcal{O}}(X;\R^l)$  is defined to be the set of all $h\in \Acal(\C,L_{\textup{cl}}^\mu (X;\R^l))$ such that
$h(z,\lambda )|_{\Gamma _\beta \times \R^l }\in L_{\textup{cl}}^\mu (X;\Gamma _\beta \times \R^l))$ for every real $\beta ,$ uniformly in compact $\beta $-intervals. In the latter notation $\textup{Re}\,z$ for $z\in \Gamma _\beta$ belongs to the parameters.\\

\nt For a Mellin amplitude function $f(r,r',z,\lambda )$ taking values in $L_{\textup{cl}}^\mu (X;\Gamma _{1/2-\gamma } \times \R^l)$ we form the associated (parameter-dependent) weighted Mellin pseudo-differential operator
\beq\label{11.E.melosc}
\op_M^\gamma (f)(\lambda )u(r)=\int_\R\int_0^\infty (r/r')^{-(1/2-\gamma +i\rho )}f(r,r',1/2-\gamma +i\rho ,\lambda )u(r')dr'/r'\dbar \rho , 
\eeq
$ \dbar \rho=(2\pi )^{-1}d\rho,$ interpreted as a Mellin oscillatory integral and representing a family of operators 
\beq\label{11.E.mpso}
C_0^\infty (\R_+,C^\infty (X))\rightarrow C^\infty (\R_+,C^\infty (X)).
 \eeq
 in $L_{\textup{cl}}^\mu (X^\wedge;\R^l).$
We can also write $$\op_M^\gamma (f)u=r^\gamma \op_M(T^{-\gamma }f) r^{-\gamma }u$$ for $\op_M(\cdot):=\op_M^0(\cdot)$ and $(T^\beta f)(z):=f(z+\beta ),\beta \in\R$ (the other variables in the latter expression have been suppressed). \\
\begin{thm}\label{11.T.cut} \textup{(}cf. \cite[Section 2.2.2]{Schu20}\textup{)}
For every $f(z,\lambda )\in L_{\textup{cl}}^\mu (X;\Gamma _\beta \times \R^l)$ for some fixed $\beta \in \R$ there exists an $h(z,\lambda )\in M^\mu _{\mathcal{O}}(X;\R^l)$ such that 
\beq\label{11.E.meid}
f=h|_{\Gamma _\beta \times \R^l} \,\,\,\mbox{\textup{mod}}\,\,\,L^{-\infty }(X;\Gamma _\beta \times \R^l),
\eeq
and $h$ is unique modulo $M^{-\infty }_{\mathcal{O}}(X;\R^l).$
\end{thm}
\begin{thm}\label{11.T.melconv} \textup{(}cf. \cite{Schu33}, or \cite[Section 2.2.2]{Schu20}\textup{)}
Let
\beq\label{11.E.asp}
\tilde{p}(r,\tilde{\rho })\in C^\infty (\overline{\R}_+,L_{\textup{cl}}^\mu (X;\R_{\tilde{\rho }})),\,p(r,\rho ):=\tilde{p}(r,r\rho ).
\eeq
Then there is an $h(r,z)\in C^\infty (\overline{\R}_+,M^\mu _{\mathcal{O}}(X)),$ unique modulo $C^\infty (\overline{\R}_+,M^{-\infty } _{\mathcal{O}}(X)),$ such that 
\beq\label{11.E.melconvt}
\Op_r(p)=\op_M^\beta (h)\,\,\,\mbox{\textup{mod}}\,\,\,L^{-\infty }(X^\wedge)
\eeq
for every $\beta \in \R.$ Conversely, for every $h(r,z)\in C^\infty (\overline{\R}_+,M^\mu _{\mathcal{O}}(X))$ there exists a $\tilde{p}(r,\tilde{\rho })\in C^\infty (\overline{\R}_+,L_{\textup{cl}}^\mu (X;\R_{\tilde{\rho }})),$ unique modulo $C^\infty (\overline{\R}_+,L^{-\infty }(X;\R_{\tilde{\rho }})),$ such that \eqref{11.E.melconvt} holds for every $\beta \in \R.$
\end{thm}
\nt The relation \eqref{11.E.asp} refers to pseudo-differential operators on $X^\wedge$ in the sense of mappings 
\eqref{11.E.mpso}.

\nt The structures of the cone algebra may be deduced by the task to express the parametrix of an elliptic operator of the form \eqref{i.degc}. For $h(r,z):=\sum_{j=0}^\mu a_j(r)z^j$ which is an element of $C^\infty (\overline{\R}_+,
M^\mu _\Ocal(X))$ we write
\beq\label{11.E.diff}
A=r^{-\mu }\op^{\gamma -n/2}_M(h)
\eeq
for any fixed $\gamma \in \R$ and $n:=\textup{dim}\,X.$ In order to briefly recall the result we 
identify $\textup{int}\,M:=M\setminus \{v\}$ with $\textup{int}\,\M=\M\setminus \partial \M$ and a collar neighbourhood of $\partial \M$ in $\M $ with $\overline{\R}_+\times X$ in the splitting of variables $(r,x).$ By a cut-off function on $\M$ we understand any real-valued $\omega \in C^\infty _0(\overline{\R}_+)$ that is equal to $1$ in a neighbourhood of $r=0.$ Given two functions $\varphi ,\varphi '\in C^\infty (\M)$ we write $\varphi \prec \varphi '$ if $\varphi '\equiv 1$ on $\textup{supp}\,\varphi .$ Now let $\omega ,\omega ',\omega ''$ be cut-off functions with $\omega ''\prec \omega \prec \omega ',$ and set
$\chi :=1-\omega ,\chi ':=1-\omega ''.$ Then every $A\in \textup{Diff}^\mu _{\textup{deg}}(M)$ can be written in the form
\beq\label{11.E.diffdec}
A=\omega r^{-\mu }\op^{\gamma -n/2}_M(h)\omega '+\chi A_{\textup{int}}\chi '
\eeq
where $A_{\textup{int}}$ in the second summand on the right is $A$ itself. In general, the pseudo-differential operators of the cone algebra  have the form
\beq\label{11.E.conedec}
A=\omega r^{-\mu }\op^{\gamma -n/2}_M(h)\omega '+\chi A_{\textup{int}}\chi ' +M+G
\eeq
for arbitrary $h\in C^\infty (\overline{\R}_+,
M^\mu _\Ocal(X)),A_{\textup{int}}\in L^\mu _{\textup{cl}}(\textup{int}\,M),$ and additional summands $M+G,$ so-called smoothing Mellin plus Green operators. The latter represent asymptotic information. In order to explain their structure we formulate weighted cone spaces and subspaces with asymptotics.\\

\nt First $\Hcal^{s,\gamma }(\R_+\times \R^n)$ for $s,\gamma \in \R$ is defined as the completion of $C_0^\infty (\R_+\times \R^n)$ with respect to the norm
\beq\label{11.E.conenorm}
\|u \| _{\Hcal^{s,\gamma }(\R_+\times \R^n)}:=\Big\{\int_{\R^n}\int_{\Gamma _{(n+1)/2-\gamma }}\langle z,\xi \rangle^{2s}|M_{r\rightarrow z}F_{x\rightarrow \xi }u(z,\xi )|^2\dbar z d\xi\Big\}^{1/2}.
\eeq
For compact $X$ we then obtain the spaces $\Hcal^{s,\gamma }(X^\wedge),X^\wedge=\R_+\times X,$ by using
\eqref{11.E.conenorm} as local representations, combined with a partition of unity. The notation with calligraphic letters indicates that the spaces contain some specific information also at $r=\infty .$ On a compact manifold $M$ with conical singularity $v$ we set 
\beq\label{11.E.conespace}
H^{s,\gamma }(M):=\{u\in H^s_{\textup{loc}}(M\setminus \{v\}):\omega u\in \Hcal^{s,\gamma }(X^\wedge)\}
\eeq
for some cut-off function $\omega .$ Recall that for convenience we identify $M\setminus \{v\}$ locally close to the singularitiy with $X^\wedge,$ i.e. we often suppress pull backs or push forwards under a corresponding map from the local model to the manifold itself. This map is kept fixed, in particular, the local splitting of variables $(r,x).$ This is the usual point of view when we formulate asymptotics for $r\rightarrow 0.$ Observe that the operator of multiplication by ${\bf{r}}^\beta, \beta \in \R,$ induces an isomorphism
\beq\label{11.E.coflk}
{\bf{r}}^\beta:H^{s,\gamma }(M)\rightarrow H^{s,\gamma +\beta }(M).
\eeq
\nt Here ${\bf{r}}^\beta $ is the weight function defined at the beginning of this section.

\nt Asymptotics of an element $u\in H^{s,\gamma }(M)$ of type $P:=\{(p_j,m_j)\}_{j\in \N}\subset \C\times \N$ (where we assume $\pi _\C P:=\{p_j\}_{j\in \N}\subset \{\textup{Re}\,z<(n+1)/2-\gamma \}$ and $\textup{Re}\,p_j\rightarrow -\infty $ as $j\rightarrow \infty $ when $\pi _\C P$ is an infinite set) means that for every $\beta \geq 0$ there is an $N=N(\beta )\in \N$ such that for any cut-off function $\omega $
\beq\label{11.E.asflk}
u-\sum_{j=0}^N\sum_{k=0}^{m_j}\omega (r)c_{jk}r^{-p_j}\textup{log}^k r\in {\bf{r}}^\beta H^{s,\gamma }(M)
\eeq 
for some coefficients $c_{jk}\in C^\infty (X).$ It can easily be proved that the coefficients are uniquely determined by $u$. The space of those $u$ will be denoted by $H_P^{s,\gamma }(M).$ It may also be interesting to control the asymptotics in terms of a condition $c_{jk}\in L_j,0\leq k\leq m_j,$ where $L_j$ are prescribed finite-dimensional subspaces of $C^\infty (X).$ In this case the asymptotic type is a corresponding sequence of triples $\{(p_j,m_j,L_j)\}_{j\in \N}.$ In the general discussion we content ourselves with the case $P:=\{(p_j,m_j)\}_{j\in \N}.$ Asymptotics including $L_j$ are a simple modification.\\

\nt For computations it is convenient to take into account also finite asymptotic expansions. We say that a $P=\{(p_j,m_j)\}_{j=0,\dots,J}$ for a $J=J(P)\in \N$ is associated with the weight data ${\bf{g}}:=(\gamma  ,\Theta )$ for a finite (so-called weight interval) $\Theta :=(\vartheta ,0]\subset \overline{\R}_-$ if $\pi _\C P:=\{p_j\}_{j=0,\dots,J}\subset \{(n+1)/2-\gamma +\vartheta <\textup{Re}\,z<(n+1)/2-\gamma \}.$ An asymptotic type $P$ is said to satisfy the shadow condition if $p\in \pi _\C P$ entails $p-j\in \pi _\C P$ for all $j\in \N$ with $\textup{Re}\,p-j>(n+1)/2-\gamma +\vartheta .$ Let $\Theta $ be finite.
The spaces
\beq\label{11.E.astflk}
\Ecal_P:=\{\sum_{j=0}^J\sum_{k=0}^{m_j}\omega (r)c_{jk}r^{-p_j}\textup{log}^k :c_{jk}\in C^\infty (X)\,\,\mbox{for all}\,\,j,k\}
\eeq 
and
\beq\label{11.E.astflkfl}
H^{s,\gamma }_\Theta (M):=\bigcap_{\varepsilon >0} H^{s,\gamma -\vartheta -\varepsilon }(M)
\eeq
are both Fr\'echet in a natural way, and we have $H^{s,\gamma }_\Theta (M)\cap \Ecal_P=\{0\}.$ We set  
\beq\label{11.E.astflkflm}
H^{s,\gamma }_P(M):=\Ecal_P+H^{s,\gamma }_\Theta (M)
\eeq
in the Fr\'echet topology of the direct sum. Note that in the above case $P:=\{(p_j,m_j)\}_{j\in \N}\subset \C\times \N$ corresponds to $\Theta =(-\infty ,0],$ and the space $H_P^{s,\gamma }(M)$ can equivalently be defined as $H_P^{s,\gamma }(M):=\bigcap _{k\in \N} H^{s,\gamma }_{P_k}(M)$ where $P_k:=\{(p,m)\in P:\textup{Re}\,p_k>(n+1)/2-\gamma -(k+1)\}$ is finite, associated with the weight data $(\gamma ,\Theta _k)$ for $\Theta _k:=(-(k+1),0],\,k\in \N.$\\

\nt An operator $G\in \bigcap _{s\in \R}\Lcal(H^{s,\gamma }(M),H^{\infty ,\gamma -\mu }(M))$ is said to be a Green operator in the cone algebra, associated with the weight data ${\bf{g}}:=(\gamma ,\gamma -\mu ,\Theta )$ if it induces continuous operators
\beq\label{11.E.green}
G:H^{s,\gamma }(M)\rightarrow H_P^{\infty ,\gamma -\mu }(M),\,\,\,G^*:H^{s,-\gamma +\mu }(M)\rightarrow H_Q^{\infty ,-\gamma }(M),
\eeq
for all $s\in \R$ and $G$-dependent asymptotic types $P,Q,$ associated with the weight data $(\gamma -\mu ,\Theta )$ and
$(-\gamma ,\Theta ),$ respectively. By $$L_G(M,{\bf{g}})\,\,\mbox{for}\,\,{\bf{g}}:=(\gamma ,\gamma -\mu ,\Theta )$$
we denote the space of all Green operators.
\nt Let us also recall the notion of smoothing Mellin operators, occurring in \eqref{11.E.conedec}. A sequence 
\beq\label{11.E.mels}
R:=\{(r_j,n_j)\}_{j\in \Z}\subset \C\times \N\,\,\,\mbox{with}\,\,\,|\textup{Re}\,r_j|\rightarrow \infty \,\,\,\mbox{as }\,\,\,|j|\rightarrow \infty 
\eeq 
is called a Mellin asymptotic type. Set $\pi _\C R:=\{r_j\}_{j\in \Z}.$ A $\chi \in C^\infty (\C)$ is called a $\pi _\C R$-excision function if $\chi (z)=0 $ for $\textup{dist}\,(z,\pi _\C R)<\varepsilon _0,$ and $\chi (z)=1 $ for $\textup{dist}\,(z,\pi _\C R)>\varepsilon _1$ for some $0<\varepsilon _0<\varepsilon _1.$ The space $M_R^{-\infty }(X)$ is defined to be the set of all $f\in \Acal(\C\setminus \pi _\C R,L^{-\infty }(X))$ such that $\chi f|_{\Gamma _\beta }\in L^{-\infty }(X;\Gamma _\beta )$ for every $\beta \in \R$ and any $\pi _\C R$-excision function $\chi  ,$ uniformly in finite $\beta $-intervals, and in addition $f$ is meromorphic with poles at the points $r_j$ of multiplicity $n_j+1$ and Laurent coefficients at $(z-r_j)^{-(k+1)},0\leq k\leq n_j,$ in $L^{-\infty }(X)$ of finite rank. Moreover, we set
\beq\label{11.E.melass}
 M^\mu _R(X):=M^\mu _\Ocal(X)+M_R^{-\infty }(X).
\eeq
\begin{rem}\label{11.R.mdec}
The decomposition \eqref{11.E.melass} is not direct since $M^\mu _\Ocal(X)\cap M_R^{-\infty }(X)=M_\Ocal^{-\infty }(X)$ is non-trivial. For any $f\in  M^\mu _R(X)$ we can produce a $g\in M_R^{-\infty }(X)$ such that $f-g\in M^\mu _\Ocal(X)$ by applying Theorem \ref{11.T.cut} to $f|_{\Gamma _\beta }\in L^\mu _{\textup{cl}}(X;\Gamma _\beta)$ for any $\beta \in \R$ with $\pi _\C R\cap \Gamma _\beta =\emptyset.$ More generally, this process applies to $f(y,z)\in C^\infty (\Omega ,M^\mu _R(X)),\Omega \subseteq \R^q$ open, and gives us a $g(y,z)\in C^\infty (\Omega ,M^{-\infty } _R(X))$ such that $f(y,z)-g(y,z)\in C^\infty (\Omega ,M^\mu _\Ocal(X)).$
\end{rem}
\begin{rem}\label{11.R.ord}
Let $f\in  M^\mu _R(X)$ and choose a $\beta \in \R$ such that $\pi _\C R\cap \Gamma _\beta =\emptyset,$ and assume that $f|_{\Gamma _\beta}\in L^{\mu -1}_{\textup{cl}}(X;\Gamma _\beta).$ Then we have $f\in  M^{\mu -1}_R(X).$
\end{rem}
\begin{thm}\label{11.T.mmu} \textup{(}cf. \cite[Section 2.1.2]{Schu2}\textup{)}
We have
\beq\label{11.E.memu}
f\in  M^\mu _R(X),g\in  M^\nu _Q(X)\Rightarrow fg\in  M^{\mu +\nu }_S(X)
\eeq
for every $R,Q$ with some resulting $S.$
\end{thm}
\nt The interpretation of the multiplication in Theorem \ref{11.T.mmu} is $z$-wise, first for all $z\in \C\setminus (\pi _\C R\cup \pi _\C Q);$ then this product extends across $\pi _\C R\cup \pi _\C Q$ to an element of $M^{\mu +\nu }_S(X).$ In particular, if $f\in M^{-\infty } _R(X)$ or $f\in M^{-\infty } _Q(X),$ then $fg\in M^{-\infty } _S(X).$
\begin{thm}\label{11.T.minv} \textup{(}cf. \cite[Section 2.4.3]{Schu20}\textup{)}
For every $m\in M^{-\infty } _R(X)$ for some Mellin asymptotic type $R$ there exists an $l\in M^{-\infty } _S(X)$ for another Mellin asymptotic type $S$ such that
\beq\label{11.E.meinv}
(1+m)^{-1}=1+l,
\eeq
i.e. $(1+m)(1+l)=(1+l)(1+m)=1$ in the sense of the multiplication of the preceding theorem.
\end{thm} 

\nt More generally we have the following result.
\begin{thm}\label{11.T.minverse}
Let $f\in M^\mu _R(X)$ and assume that $f|_{\Gamma _\beta }\in L^\mu _{\textup{cl}}(X,\Gamma _\beta )$ is parameter-dependent elliptic. Then we have $f^{-1}\in M^{-\mu }_S(X)$ for some Mellin asymptotic type $S.$
\end{thm}
\begin{proof}
Let $g_\beta \in L^{-\mu} _{\textup{cl}}(X,\Gamma _\beta )$ be a parameter-dependent parametrix of $f|_{\Gamma _\beta }.$ Applying Theorem \ref{11.T.cut} for $l=0$ we find an $h\in M^{-\mu }_{\mathcal{O}}(X)$ such that $g_\beta=h|_{\Gamma _\beta }$ modulo $L^{-\infty }(X;\Gamma _\beta ).$ It follows that
$$(hf)|_{\Gamma _\beta}=1+r_\beta \,\,\,\mbox{for some}\,\,\,r_\beta \in L^{-\infty }(X;\Gamma _\beta ).$$
However, then we know that $hf=1+m$ for some $m\in M^{-\infty } _R(X).$ From Theorem \ref{11.T.minv} it follows that $(1+l)hf=1,$ i.e. $(1+l)h=f^{-1}$ is as claimed.
\end{proof}
 \nt A smoothing Mellin operator in the cone calculus associated with the weight data ${\bf{g}}:=(\gamma ,\gamma -\mu ,\Theta )$ for $\Theta =(-(k+1),0],k\in \N,$ is defined as
\beq\label{11.E.smme}
M:= r^{-\mu }\sum_{j=0}^k\omega r^j\op_M^{\gamma_j -n/2}(f_j)\omega '
\eeq
for cut-off functions $\omega ,\omega ',$ elements $f_j\in M^{-\infty }_{R_j}(X)$ for Mellin asymptotic types $R_j,$ and weights $\gamma _j$ with $\gamma -j\leq \gamma _j\leq \gamma $ and $\pi _\C R_j\cap \Gamma _{(n+1)/2-\gamma _j}=\emptyset$ for all $j.$ Let
$$L_{M+G}(M,{\bf{g}})\,\,\mbox{for}\,\,{\bf{g}}:= (\gamma ,\gamma -\mu ,\Theta ),\,\Theta =(-(k+1),0],$$
denote the space of all operators $M+G$ for arbitrary $G\in L_G(M,{\bf{g}})$ and operators \eqref{11.E.smme} for arbitrary $f_j,\,j=0,\dots,k,$ for any fixed choice of cut-off functions $\omega ,\omega '.$\\\\
Now we have all ingredients of \eqref{11.E.conedec}. Let
\beq\label{11.E.conmme}
L^\mu (M,{\bf{g}})\,\,\,\mbox{for}\,\,\,{\bf{g}}=(\gamma ,\gamma -\mu ,\Theta ),\,\Theta =(-(k+1),0],
\eeq
denote the space of all operators of the form \eqref{11.E.conedec} for arbitrary $h(r,z)\in C^\infty (\overline{\R}_+,\newline M^\mu _{\mathcal{O}}(X)),$ $A_{\textup{int}}\in L^\mu _{\textup{cl}}(s_0(M)),$ and $M+G\in L_{M+G}(M,{\bf{g}}).$ For $\Theta =(-\infty ,0]$ we set $L^\mu (M,{\bf{g}}):=\bigcap _{k\in \N}L^\mu (M,(\gamma ,\gamma -\mu ,(-(k+1),0])).$ These spaces constitute the cone algebra over the compact manifold $M$ with conical singularity $v.$ The cone algebra in this form has been introduced in \cite{Remp7}.
\begin{rem}
The space $\textup{Diff}^\mu _{\textup{deg}}(M)$ belongs to $L^\mu (M,{\bf{g}})$ for every weight $\gamma $ and $\Theta =(-\infty ,0].$
\end{rem}
\begin{rem}\label{11.R.sm}
We have
\beq\label{11.E.sm}
L^\mu (M,{\bf{g}})\cap L^{-\infty }(s_0(M))=L_{M+G}(M,{\bf{g}}).
\eeq 
\end{rem} 
\nt In fact, if $A$ belongs to the left hand side of \eqref{11.E.sm} then all parameter-dependent homogeneous components of the local (over $X$) symbol of $h(r,z)\in C^{\infty }(\overline{\R}_+,L^{\mu }_{\textup{cl}}(X;\newline\Gamma _{(n+1)/2-\gamma }))$ vanish for $r>0.$ Since those are smooth up to $r=0$ they vanish up to $0$, and it follows that $h(r,z)\in C^{\infty }(\overline{\R}_+,L^{-\infty  }(X;\Gamma _{(n+1)/2-\gamma })),$ more precisely, $h(r,z)\in C^{\infty }(\overline{\R}_+,M_{\mathcal{O}}^{-\infty }(X)),$ and for such an $h$ we have $\omega r^{-\mu }\op_M^{\gamma -n/2}(h)\omega '\in L_{M+G}(M,{\bf{g}}).$ 
\nt Every $A\in L^\mu (M,{\bf{g}})$ induces continuous operators 
\beq\label{11.E.conta}
A:H^{s,\gamma }(M)\rightarrow H^{s-\mu ,\gamma -\mu }(M),\,\, H_P^{s,\gamma }(M)\rightarrow H_Q^{s-\mu ,\gamma -\mu }(M)
\eeq
for every $s\in \R$ and every asymptotic type $P$ associated with the weight data $(\gamma  ,\Theta )$ for some resulting $Q$ associated with the weight data $(\gamma -\mu  ,\Theta ).$ 
More details may be found, in \cite{Schu2}, see also \cite{Schu20}. \\

\nt The composition refers to a bilinear operation
\beq\label{11.E.sccom}
L^\mu (M,(\gamma -\nu,\gamma -(\mu +\nu),\Theta ))\times L^\nu  (M,(\gamma ,\gamma -\nu ,\Theta ))\rightarrow L^{\mu +\nu } (M,(\gamma,\gamma -(\mu +\nu ),\Theta ))
\eeq
Recall, in particular, that the spaces
\beq\label{11.E.conm+g}
L _{M+G}(M,{\bf{g}})\,\,\,\mbox{and}\,\,\,L _G(M,{\bf{g}})
\eeq
of smoothing Mellin plus Green operators $M+G,$ and $G,$ respectively, form subalgebras, where $AB$ belongs to $L_{M+G}\,\,(L _G)$ as soon as $A$ or $B$ belongs to $L_{M+G}\,\,(L _G).$ We have also a formal adjoint
\beq\label{11.E.sccomad}
L^\mu (M,(\gamma ,\gamma -\mu ,\Theta ))\rightarrow L^\mu (M,( -\gamma +\mu ,-\gamma,\Theta ))
\eeq
defined by $(Au,v)_{H^{0,0}(M)}=(u,A^*v)_{H^{0,0}(M)}$ for all $u,v\in C_0^\infty (M\setminus \{v\}),$ and $L_{M+G}\,\,(L _G)$ are preserved under formal adjoints.\\

\nt Let us now express the principal symbolic structure $\sigma (A)=(\sigma _0(A),\sigma _1(A))$ of operators $A\in L^\mu (M,(\gamma ,\gamma -\mu ,\Theta )).$ The relation 
$$L^\mu (M,(\gamma ,\gamma -\mu ,\Theta ))\subset L^\mu _{\textup{cl}}(M\setminus \{v\})$$
gives rise to the standard homogeneous principal symbol $\sigma _0(A)$ of order $\mu  $ of the operator as an invariantly defined function in $C^\infty (T^*(s_0(M))\setminus 0),s_0(M)=M\setminus \{v\}.$ In addition, as noted in the introduction we have the rescaled principal symbol \eqref{i.E.redy} close to $s_1(M)=v.$ The principal conormal symbol $\sigma _1(A)$ is defined as the operator function
\beq\label{11.E.sccono}
\sigma _1(A)(z):=h(0,z)+f_0(z):H^s(X)\rightarrow H^{s-\mu }(X)
\eeq
for $z\in \Gamma _{(n+1)/2-\gamma },$ cf. the notation in connection with \eqref{11.E.conedec} and \eqref{11.E.smme}. The behaviour of $\sigma (\cdot)$ under compositions $(A,B)\rightarrow AB$ in the notation of \eqref{11.E.sccom} is as follows:
\beq\label{11.E.sycom}
\sigma _0(AB)=\sigma _0(A)\sigma _0(B),\,\,\sigma _1(AB)=(T^\nu \sigma _1(A))\sigma _1(B)
\eeq
where $(T^\beta f)(z)=f(z+\beta ).$

\nt Let us also recall the Mellin translation product, cf. \cite{Schu20}, referring to the lower order conormal symbols
$$(\sigma _1^{\mu -j}(A)(z))_{j=0,1,2,\dots}\,\,\,\mbox{for}\,\,\,\sigma _1^{\mu -j}(A)(z):=1/j!(\partial ^j_rh)(0,z)+f_j(z),$$ 
$\sigma _1^\mu (A)(z):=\sigma _1 (A)(z).$
\begin{rem}\label{11.R.mtransl}
We have
\beq
\sigma _1^{\mu +\nu -l}(AB)=\sum_{i+j=l}(T^{\nu -j}\sigma _1^{\mu -i}(A))\sigma _1^{\nu -j}(B),
\eeq
$l=0,1,2,\dots.$
\end{rem}
\subsection{Ellipticity in the cone algebra} 
\begin{defn}\label{12.D.ell}
Let $A\in L^\mu (M,{\bf{g}}), {\bf{g}}=(\gamma ,\gamma -\mu ,\Theta ).$
\begin{itemize}
\item[\textup{(i)}] The operator $A$ is called $\sigma _0$-elliptic, if $\sigma _0(A)$ never vanishes on $T^*(s_0(M))\setminus 0,$ and if in addition close to $s_1(M)$ we have $\tilde{\sigma }_0(A)(r,x,\rho ,\xi )\neq 0$ for $(\rho ,\xi  )\neq 0,$ up to $r=0.$
\item[\textup{(ii)}] The operator $A$ is called $\sigma _1$-elliptic if \eqref{11.E.sccono} is a family of isomorphisms for all $z\in \Gamma _{(n+1)/2-\gamma }$ for any $s\in \R.$
\end{itemize}
We call $A$ elliptic if it satisfies the conditions $\textup{(i)},\textup{(ii)}.$
\end{defn}
\begin{rem}\label{12.R.p}
If $A$ is elliptic we find a $\tilde{p}(r,\tilde{\rho })\in C^\infty (\overline{\R}_+,L^\mu _{\textup{cl}}(X;\R_{\tilde{\rho }})$ with values in parameter-dependent elliptic operators on $X$ (with parameter $\tilde{\rho }$) such that $h(r,z)\in C^\infty (\overline{\R}_+,M^\mu _{\mathcal{O}}(X))$ in the representation \eqref{11.E.conedec}
is associated with $\tilde{p}$ via Theorem \ref{11.T.melconv}.
\end{rem}
\begin{rem}\label{12.R.coel}
It may be difficult to explicitly check the condition \textup{(ii)} of Definition \ref{12.D.ell}; the non-bijectivity points of \eqref{11.E.sccono} in the complex plane are a kind of non-linear eigenvalues of $\sigma _1(A)$ as operators globally on $X.$ Often in applications we have $A\in \textup{Diff}^\mu _{\textup{deg}}(M),$ and $X$ may have several connected components $X_1,\dots,X_N.$ In this case, because of the locality of differential operators, $\sigma _1(A)(z):H^s(X)\rightarrow H^{s-\mu }(X)$ is bijective if and only if $\sigma _1(A)(z):H^s(X_j)\rightarrow H^{s-\mu }(X_j)$ is bijective for every $j=1,\dots,N.$
\end{rem}
\begin{thm}\label{12.T.par}
Let $M$ be a manifold with conical singularities. An elliptic operator $A\in L^\mu (M,{\bf{g}})$ for ${\bf{g}}=(\gamma ,\gamma -\mu ,\Theta )$ has a parametrix $P\in L^{-\mu} (M,{\bf{g}}^{-1})$ for ${\bf{g}}^{-1}=(\gamma -\mu ,\gamma ,\Theta )$ in the sense that $PA-1\in L_G(M,{\bf{g}}_l),AP-1\in L_G(M,{\bf{g}}_r)$ for ${\bf{g}}_l=(\gamma ,\gamma ,\Theta ),{\bf{g}}_r=(\gamma -\mu, \gamma -\mu, \Theta ).$
\end{thm}
\begin{proof}
Applying Remark \ref{12.R.p} we choose a $\tilde{p}(r,\tilde{\rho })\in C^\infty (\overline{\R}_+,L^\mu _{\textup{cl}}(X;\R_{\tilde{\rho }}))$ with values in parameter-dependent elliptic operators, and set $p(r,\rho ):=\tilde{p}(r,r\rho )$. Near the conical point in the spitting of variables $(r,x)\in X^\wedge$\,\, we have $\omega A\omega '=\omega r^{-\mu }\Op_r(p)\omega '$ modulo $L^{-\infty }(X^\wedge).$ Let us choose an operator function $\tilde{p}^{(-1)}(r,\tilde{\rho })\in C^\infty (\overline{\R}_+,L^{-\mu }_{\textup{cl}} (X;\R_{\tilde{\rho }}))$ such that
$p^{(-1)}(r,\rho ):=\tilde{p}^{(-1)}(r,r\rho )$ satisfies the relations
$$r^\mu p^{(-1)}(r,\rho )\sharp _r r^{-\mu }p(r,\rho )\sim 1,\,\,r^{-\mu }p(r,\rho )\sharp _r r^\mu p^{(-1)}(r,\rho )\sim 1$$
where $\sharp _r$ means the Leibniz product between the respective symbols (cf. also Lemma \ref{23.L.p1} below). By virtue of Theorem \ref{11.T.melconv} there is an $$h^{(-1)}(r,z)\in C^\infty (\overline{\R}_+, M^{-\mu }_{\mathcal{O}}(X))$$ such that $\omega r^\mu \Op_r(p^{(-1)})\omega '=\omega r^\mu \op_M^{\gamma -\mu -n/2}(h^{(-1)})\omega '$ modulo a smoothing remainder.\\
\nt Without loss of generality in the representation \eqref{11.E.conedec} we may assume that $A_{\textup{int}}$ coincides with $A.$ Then $A_{\textup{int}}$ is elliptic on $s_0(M )$ in the standard sense, and we find a parametrix $A_{\textup{int}}^{(-1)}\in L^{-\mu }_{\textup{cl}}(s_0(M )).$ Now we set $A^{(-1)}:=r^\mu \omega \op_M^{\gamma -\mu -n/2}(h^{(-1)})\omega '+\chi A_{\textup{int}}^{(-1)}\chi '.$ Then we have $A^{(-1)}\in L^{-\mu }(M,{\bf{g}}^{-1})$ and $A^{(-1)}A-1\in L^0(M,{\bf{g}}_l)\cap L^{-\infty }(s_0(M)).$ Thus by virtue of Remark \ref{11.R.sm} we have $A^{(-1)}A-1\in L_{M+G}(M,{\bf{g}}_l).$ Next we construct an operator $B\in L_{M+G}(M,{\bf{g}}_l)$ such that
\beq\label{12.E.pmB} 
(A^{(-1)}+B)A=1 \,\,\textup{mod}\,\, L_G(M,{\bf{g}}_l).
\eeq
 We first find an $l\in M_R^{-\infty }(X)$ for some Mellin asymptotic type $R$ such that 
\beq
(h^{-1}(0,z+\mu )+l(z+\mu ))(h(0,z)+f_0(z))=1,
\eeq
cf. the notation in \eqref{11.E.smme}. We use the fact that $(h(0,z)+f_0(z))^{-1}=h^{-1}(0,z+\mu )+m(z+\mu )$ for some $m(z)\in M_S^{-\infty }(X),$ cf. Theorem \ref{11.T.minverse}. Thus it suffices to set $l(z)=m(z).$ For $B_0:= r^\mu \omega \op_M^{\gamma -\mu -n/2}(l)\omega '$ we obtain 
$$(A^{(-1)}+B_0)A=1+C\,\,\,\mbox{for a}\,\,\,C\in L_{M+G}(M,{\bf{g}}_l)$$
and $\sigma _1(C)=0$. Now we use the fact that $(\sum_{j=0}^k(-1)^jC^j)(1+C)=1+G$ for some $G\in L_G(M,{\bf{g}}).$ It follows that
\beq\label{12.E.pm}
(\sum_{j=0}^k(-1)^jC^j)(A^{(-1)}+B_0)=:P
\eeq
is a left parametrix of $A.$ In other words, we have constructed $B:=(\sum_{j=0}^k(-1)^jC^j)\newline B_0+(\sum_{j=1}^k(-1)^jC^j)A^{(-1)}$ mentioned before in \eqref{12.E.pmB}. In an analogous manner we obtain a right parametrix with a left-over term in $L_G(M,{\bf{g}}_r).$ Thus for simple algebraic reasons \eqref{12.E.pm} is a two-sided parametrix.
\end{proof}
\begin{thm}\label{13.T.as}
Let $A\in L^\mu (M,{\bf{g}})$ for ${\bf{g}}=(\gamma ,\gamma -\mu ,\Theta )$ be elliptic. Then $Au=f$ for $u\in H^{-\infty ,\gamma }(M)$ and  $f\in H^{s-\mu  ,\gamma -\mu }(M),s\in \R,$ entails $u\in H^{s ,\gamma }(M).$ If in addition $f\in H_S^{s-\mu  ,\gamma -\mu }(M)$ for an asymptotic type $S$ associated with $(\gamma -\mu ,\Theta )$ it follows that $u\in H_R^{s ,\gamma }(M)$ for a resulting asymptotic type $R$ associated with $(\gamma ,\Theta ).$
\end{thm}
\begin{proof}
Let $P\in L^{-\mu }(M,{\bf{g}}^{-1})$ a parametrix of $A.$ Then $Au=f$ implies $PAu=(1+C)u=Pf,$ i.e. $u=Pf-Cu,$ and it suffices to apply \eqref{11.E.conta}.
\end{proof}
\subsection{Asymptotic parametrices in the conical case} 
An operator $A\in L^\mu (M,{\bf{g}})$ for ${\bf{g}}=(\gamma ,\gamma -\mu ,\Theta )$ is said to be flat of order $N \in \N$ if ${\bf{r}}^{-N}A\in L^\mu (M,{\bf{g}}).$
\begin{prop}\label{13.P.dev}
Every $A\in L^\mu (M,{\bf{g}}),{\bf{g}}=(\gamma ,\gamma -\mu ,\Theta ),$ can be written in the form
\beq\label{13.E.cdev}
A=\sum_{i=0}^N{\bf{r}}^iA_i\,\,\,\textup{mod}\,\,L_G(M,{\bf{g}})
\eeq
for any $N\in \N,$ for suitable $A_i\in L^\mu (M,{\bf{g}});$ in particular, ${\bf{r}}^iA_i$ is flat of order $i,\,i=0,\dots,N.$
\end{prop}
\begin{proof}
By definition we have $A$ in the form \eqref{11.E.conedec}. The operator $\chi A_{\textup{int}}\chi '$ is obviously flat of any order. Moreover, applying the Taylor formula to $h(r,z)$ at $r=0$ we obtain a representation of $\omega r^{-\mu }\op_M^{\gamma -n/2}(h)\omega '$ of the desired form. So it remains to consider $M $ which is of the form \eqref{11.E.smme}. In the case $\gamma _j=\gamma $ for all $j$ the claimed representation is evident. For general $\gamma _j$ we have at least $\gamma _0=\gamma .$ So it remains to consider the summands for $j>0.$ Let
$$M_j:=r^{-\mu} \omega r^j\op_M^{\gamma _j-n/2}(f_j)\omega '.$$
If $\gamma _j<\gamma $ for all (or some) $j\geq 1$ we set $\tilde{M}_0:= M_0+M_1,$ and form
$$\tilde{M}_{j-1}:=r^{-\mu }\omega r^j\op_M^{\tilde{\gamma}_j -n/2}\newline (f_j)\omega ',j\geq 2,$$
using the fact that there are weights $\tilde{\gamma }_j$ such that $\gamma -1\leq \tilde{\gamma}_j \leq \gamma $ with $\pi _\C R_j\cap \Gamma _{(n+1)/2-\tilde{\gamma}_j}=\emptyset.$ Then for 
$G_j:=\tilde{M}_{j-1}-M_j\in L_G(M,{\bf{g}}),j\geq 2,$
and for $G_1\equiv 0$ it follows that 
$$\sum_{j=0}^kM_j=\sum_{i=0}^{k-1}\tilde{M}_i-G_{i+1}$$
where ${\bf{r}}^{-i}\tilde{M}_i\in L_{M+G}(M,{\bf{g}}).$
\end{proof}
\begin{prop}\label{13.P.devcom}
Let $A\in L^\mu (M,{\bf{g}}),{\bf{g}}=(\gamma ,\gamma -\mu ,\Theta ),$ be of the form
\beq\label{13.E.aspaflatg}
A=\omega r^{-\mu }\op^{\gamma -n/2}_M(h)\omega '+\chi A_{\textup{int}}\chi ' +G
\eeq
with notation as in \eqref{11.E.conedec}, and let $\beta \geq 0.$ Then there is an $A_\beta \in L^\mu (M,{\bf{g}})$  and a $G_\beta \in L _G(M,{\bf{g}})$ such that $A{\bf{r}}^\beta -{\bf{r}}^\beta A_\beta =G_\beta .$ 
\end{prop}
\begin{proof}
First we have $G_\beta :=G{\bf{r}}^\beta \in L _G(M,{\bf{g}}).$ Moreover, $\chi A_{\textup{int}}\chi '{\bf{r}}^\beta ={\bf{r}}^\beta \chi A_{\textup{int},\beta }\chi '$ for $A_{\textup{int},\beta } :={\bf{r}}^{-\beta}\chi A_{\textup{int} }{\bf{r}}^\beta$ belongs to $L^\mu (M,{\bf{g}}).$ Let us write $A$ in the form \eqref{11.E.conedec}. Concerning the first summand on the right of \eqref{13.E.aspaflatg} we observe $r^\beta r^{-\beta  }\op_M^{\gamma -n/2}(h)r^\beta  \newline
=r^\beta \op_M^{\gamma -n/2-\beta }(T^{-\beta  }h)=r^\beta \op_M^{\gamma -n/2}(T^{-\beta  }h);$ the latter relation employs Cauchy's theorem, cf. \cite[Propsition 2.3.69]{Schu20}. Thus $\omega r^{-\mu }\op^{\gamma -n/2}_M(h)\omega '{\bf{r}}^\beta={\bf{r}}^\beta\omega r^{-\mu }\op^{\gamma -n/2}_M \newline (T^{-\beta  }h)\omega ',$ but $\omega r^{-\mu }\op^{\gamma -n/2}_M(T^{-\beta  }h)\omega '$ belongs to $L^\mu (M,{\bf{g}}).$
\end{proof}
\nt Comparing \eqref{13.E.aspaflatg} with \eqref{11.E.conedec} it remains to consider the commutation of ${\bf{r}}^\beta$ through the smoothing Mellin operator $M.$ Clearly, close to $r=0$ we do not make any difference between ${\bf{r}}$ and $r.$
\begin{prop}\label{13.P.dmelcom}
Let $M\in L^\mu (M,{\bf{g}}),{\bf{g}}=(\gamma ,\gamma -\mu ,\Theta ),$ be a smoothing Mellin operator
\beq\label{13.E.aspaflme}
M:=M_0+M_1\,\,\mbox{for}\,\,M_0:=r^{-\mu }\omega\op_M^{\gamma -n/2}(f_0)\omega ',M_1:=r^{-\mu }\sum_{j=1}^k\omega r^j\op_M^{\gamma_j -n/2}(f_j)\omega ',
\eeq
for Mellin symbols $f_j\in M^{-\infty }_{R_j}(X)$ for Mellin asymptotic types $R_j,$ and weights $\gamma _j$ with $\gamma -j\leq \gamma _j\leq \gamma $ and $\pi _\C R_j\cap \Gamma _{(n+1)/2-\gamma _j}=\emptyset$ for all $j,$ cf. \eqref{11.E.smme}, and let $\beta \geq 0.$
Then there is an $M_{1,\beta} \in L^\mu _{M+G}(M,{\bf{g}})$  and a $G_{1,\beta} \in L _G(M,{\bf{g}})$ such that $M_1{\bf{r}}^\beta -{\bf{r}}^\beta M_{1,\beta} =G_{1,\beta} .$ Analogously, under the condition $\pi _\C R_0\cap \Gamma _{(n+1)/2-(\gamma +\beta )}=\emptyset$ there is an $M_{0,\beta} \in L^\mu _{M+G}(M,{\bf{g}})$  and a $G_{0,\beta} \in L _G(M,{\bf{g}})$ such that $M_0{\bf{r}}^\beta -{\bf{r}}^\beta M_{0,\beta} =G_{0,\beta} .$ In general, for $\beta >0$ and every $0<\varepsilon <\beta $ we have
\beq\label{13.E.aspaflmeap}
M_0{\bf{r}}^\beta={\bf{r}}^{\beta-\varepsilon }\omega r^{-\mu }\op_M^{\gamma -n/2}(T^{-\beta +\varepsilon }f_0){\bf{r}}^\varepsilon \omega '+G_{0,\beta -\varepsilon },
\eeq
for a $G_{0,\beta -\varepsilon }\in  L _G(M,{\bf{g}}).$
\end{prop}
\begin{proof}
The assertion for $M_0$ is contained in \cite[Propsition 2.3.69]{Schu20}. For $M_1$ we can argue in a similar manner. Here, because of the extra $r^j$-powers for $j>0$ the condition concerning weight lines disappears.
\end{proof}
\nt Let us now turn to asymptotic parametrices. Let $A\in L^\mu (M,{\bf{g}})$ be elliptic, and let $P\in L^{-\mu }(M,{\bf{g}}^{-1})$ be a parameterix, cf. Theorem \ref{12.T.par}. Then applying Proposition \ref{13.P.dev} to $P$ we obtain an expansion
\beq\label{13.E.aspa}
P=\sum_{m=0}^N{\bf{r}}^mP_m  \,\,\,\textup{mod}\,L_G(M,{\bf{g}}^{-1})
\eeq
for every $N\in \N.$ Let us call \eqref{13.E.aspa} an asymptotic parametrix of $A.$ Writing $A$ itself in the form \eqref{13.E.cdev} the operator $A_0$ is elliptic in $L^\mu (M,{\bf{g}}),$ and $P_0\in L^{-\mu }(M,{\bf{g}}^{-1})$ in \eqref{13.E.aspa} can be taken as a parametrix of $A_0,$  obtained by Theorem \ref{12.T.par}. The construction of $P_0$ is easier than that of $P$ itself because $P_0$ concerns the case of constant coefficients with respect to $r.$ The idea of asymptotic paramterices is to obtain an expansion like \eqref{13.E.aspa} by an iterative process, starting with the ansatz
\beq\label{13.E.aspit}
\big(\sum_{m=0}^N{\bf{r}}^mP_m \big)\big(\sum_{i=0}^N{\bf{r}}^iA_i \big)\sim 1,
\eeq
and, analogously,
\beq\label{13.E.aspitr}
\big(\sum_{i=0}^N{\bf{r}}^iA_i \big)\big(\sum_{m=0}^N{\bf{r}}^mP_m \big)\sim 1
\eeq
where $\sim $ indicates equality modulo a Green plus flat remainder. Let us discuss the construction of \eqref{13.E.aspit}; the considerations for \eqref{13.E.aspitr} are similar. The operators ${\bf{r}}^mP_m$ that we intend to obtain from successively solving \eqref{13.E.aspit} are not automatically the same as those in \eqref{13.E.aspa}. Therefore, in this procedure it makes sense first to replace ${\bf{r}}^mP_m$ by $\boldsymbol{P_m}.$  However, we set $\boldsymbol{P_0}:=P_0$ which is the known parametrix of $A_0.$ For $\boldsymbol{P_1}$ we define
\beq\label{13.E.asp1}
\boldsymbol{P_1}:=-P_0{\bf{r}}A_1P_0.
\eeq
This operator satisfies the relation
\beq\label{13.E.asp11}
\boldsymbol{P_1}A_0+P_0{\bf{r}}A_1=-P_0{\bf{r}}A_1C_l\in L_G(M,{\bf{g}}_l)
\eeq
where $C_l:=P_0A_0-1\in L_G(M,{\bf{g}}_l).$ In general, for $ j>0,$ motivated by the desirable identity
$\sum_{m+i=j}\boldsymbol{P_m}{\bf{r}}^iA_i=0,\,\,\,\mbox{or}\,\,\,\boldsymbol{P_j}A_0=-\sum_{m+i=j,m<j}\boldsymbol{P_m}{\bf{r}}^iA_i,$ we set
\beq\label{13.E.asp1mj}
\boldsymbol{P_j}:=-\big(\sum_{m+i=j,m<j}\boldsymbol{P_m}{\bf{r}}^iA_i\big)P_0 \in L^{-\mu }(M,{\bf{g}}^{-1}),
\eeq
and obtain
\beq\label{13.E.asp1mj3}
\sum_{m+i=j}\boldsymbol{P_m}{\bf{r}}^iA_i=-\big(\sum_{m+i=j,m<j}\boldsymbol{P_m}{\bf{r}}^iA_i\big)C_l\in L^{-\infty }(M,{\bf{g}}_l)
\eeq 
for any $j.$ The operators \eqref{13.E.asp1mj} can be expressed in terms of $P_0$ and ${\bf{r}}^iA_i,$ for $i=1,\dots,j,$ alone. For instance, we have
\beq\label{13.E.aspber2}
\boldsymbol{P_2}=-P_0{\bf{r}}^2A_2P_0+P_0{\bf{r}}A_1P_0{\bf{r}}A_1P_0,
\eeq 
\beq\label{13.E.aspber3}
\boldsymbol{P_3}=-P_0{\bf{r}}A_1P_0{\bf{r}}A_1P_0{\bf{r}}A_1P_0+P_0{\bf{r}}^2A_2P_0{\bf{r}}A_1P_0+P_0{\bf{r}}A_1P_0{\bf{r}}^2A_2P_0-P_0{\bf{r}}^3A_3P_0,
\eeq
etc. These representations allow us to establish a scheme of successively computing the asymptotics of solutions in the sense of Theorem \ref{13.T.as} only by using the structure of $P_0.$ In the applications mentioned at the beginning the operator $A$ is a differential operator. In this case also $A_m$ are differential operators in the cone calculus, and there are differential operators $A_{m,i}$ of analogous structure, such that
\beq\label{13.E.commd}
A_m{\bf{r}}^i-{\bf{r}}^iA_{m,i}=0
\eeq
for every $i.$ Let us assume for the moment that 
\beq\label{13.E.comm}
P_0{\bf{r}}^i-{\bf{r}}^iP_{0,i}=G_i ,
\eeq
for every $i,$ for remainders $G_i\in L_G (M, {\bf{g}}^{-1})$ of finite rank (which is the case when the operator $M_0$ contained in $P_0$ satisfies the corresponding condition of Proposition \ref{13.P.dmelcom}, namely, $\pi _\C R_0\cap \Gamma _{(n+1)/2-(\gamma +i )}=\emptyset$). Then all ${\bf{r}}$-powers contained on the right of \eqref{13.E.asp1mj} can be completely commuted through the other factors, on the expense of Green remainders, and it follows that
\beq\label{13.E.commexj}
\boldsymbol{P}_j={\bf{r}}^jP_j+ D_j \,\,\mbox{for suitable}\,\, P_j\in L^{-\mu }(M,{\bf{g}}^{-1})
\eeq
and finite rank operators $D_j\in L_G (M, {\bf{g}}^{-1}).$ For instance, we obtain 
\beq\label{13.E.commex}
\boldsymbol{P}_1=-({\bf{r}}P_{0,1}+G_1)A_1P_0={\bf{r}}P_1+D_1
\eeq
for $P_1=-P_{0,1}A_1P_0\in L^{-\mu }(M,{\bf{g}}^{-1})$ and $D_1=-G_1A_1P_0\in L_G (M, {\bf{g}}^{-1}),$ etc. In this case we take the sum 
\beq\label{13.E.commaspa}
P_{\textup{as}}:=\sum_{m=0}^N{\bf{r}}^mP_m 
\eeq
for any fixed $N$ as an asymptotic parametrix of $A.$
Finally, if the condition $\pi _\C R_0\cap \Gamma _{(n+1)/2-(\gamma +i )}=\emptyset$ is not satisfied for all $i,$ then a consequence of Proposition \ref{13.P.dmelcom} is that for every $\varepsilon >0$ we find operators of the form
\beq\label{13.E.commexje}
\boldsymbol{P}_{j,\varepsilon }={\bf{r}}^{j-\varepsilon }P_{j,\varepsilon }+ D_{j,\varepsilon }\,\,\mbox{for suitable}\,\, P_{j,\varepsilon }\in L^{-\mu }(M,{\bf{g}}^{-1}) ,j\geq 1,
\eeq
and finite rank operators $D_{j,\varepsilon } \in L_G (M, {\bf{g}}^{-1}),$ and $\boldsymbol{P}_0:=P_0,$ such that 
\beq\label{13.E.commaspaep}
P_{\textup{as}}:=P_0+\sum_{m=1}^N{\bf{r}}^{m-\varepsilon }P_{m,\varepsilon }
\eeq
can play the role of an asymptotic parametrix. 
\begin{thm}\label{13.P.as}
Let $A\in L^\mu (M,{\bf{g}})$ be an elliptic operator. Then for any fixed $N\in \N,N\geq 1,$ the asymptotic parametrix $P_{\textup{as}}$ can be chosen in such a way that 
\beq\label{13.E.copas}
P_{\textup{as}}A=1+C+F
\eeq
for an $F\in L^\mu (M,{\bf{g}}_l)$ which is flat of order $N-1 $ and a $C\in L_G(M,{\bf{g}}_l).$
\end{thm}
\begin{proof}
Let us first consider the case that $A$ is a differential operator. We already constructed $P_{\textup{as}},$ and it remains to verify the relation \eqref{13.E.copas}. First, by construction we have
\beq\label{13.E.copassm}
\sum_{j=0}^N\sum_{m+i=j}\boldsymbol{P_m}{\bf{r}}^iA_i=1+G
\eeq 
for some $G\in L_G(M,{\bf{g}}_l).$ If $P_{\textup{as}}$ has the form $\sum_{m=0}^N{\bf{r}}^mP_m,$ then $P_{\textup{as}}A$ coincides with \eqref{13.E.copassm} modulo a remainder  $D+\sum_{m+i>N}{\bf{r}}^mP_m{\bf{r}}^iA_i$ for a $D\in L_G(M,{\bf{g}}_l).$ The summand ${\bf{r}}^mP_m{\bf{r}}^iA_i$ contains $m+i$ factors ${\bf{r}};$ those can be commuted through the other factors to the left (on the expense of Green remainders) which gives us at the end at least a factor ${\bf{r}}^N$ from the left. In other words, we obtain the relation \eqref{13.E.copas}. For $P_{\textup{as}}$ in the form \eqref{13.E.commaspaep} we can argue in a similar manner, but we may obtain on the left ${\bf{r}}^{N-\varepsilon }$ instead of ${\bf{r}}^N.$ Therefore, since $N$ is arbitrary anyway, we may be satisfied with the exponent $N-1$ in both cases.\\
The construction easily extends to any elliptic (pseudo-differential) operator in the cone algebra. The only change is that the commutation relations \eqref{13.E.commd} have to be replaced by relations of the kind $A_m{\bf{r}}^i-{\bf{r}}^{i-\varepsilon }A_{m,i,\varepsilon }=C_{i,\varepsilon }$ for any sufficiently small $\varepsilon >0$ and a finite rank Green operator $C_{i,\varepsilon }.$
\end{proof}
\begin{cor}\label{13.C.as}
Theorem \textup{\ref{13.P.as}} allows us to express asymptotics of solutions to an elliptic equation $Au=f\in H^{s-\mu ,\gamma -\mu }_Q(M),$ only by using a parametrix $P_0$ of $A_0$ and the commutation relations that are involved in the construction of $P_{\textup{as}}.$ From \eqref{13.E.copas} it follows that $P_{\textup{as}}Au=(1+C+F)u=P_{\textup{as}}f\in H_S^{s,\gamma }(M)$ for some asymptotic type $S.$ For sufficiently large $N$ \textup{(}depending on the weight interval $\Theta =(-(k+1),0])$ we obtain $Fu\in H_\Theta ^{s,\gamma }(M).$ Morover, $Cu\in H_R^{\infty ,\gamma }(M)$ for another asymptotic type $R$ depending on $C$ gives us finally $u\in H_P^{s ,\gamma }(M)$ for a resulting asymptotic type $P.$ We tacitly employed the non-asymptotic part of Theorem \textup{\ref{13.T.as}}, i.e. that we already know $u\in H^{s ,\gamma }(M),$ and the mapping properties \eqref{11.E.conta} and \eqref{11.E.green}.
\end{cor}

\section{Elliptic edge operators} 

\subsection{Edge-degenerate operators} 
We first develop some features of the edge pseudo-differential calculus including its symbolic structure which is designed to express parametrices of elliptic elements, especially parametrices of elliptic edge-degenerate differential operators, cf. \eqref{iedeg}. A manifold $M$ with edge $s_1(M)$ can be represented as a quotient space $\mathbb{M}/\sim $ for the so-called stretched manifold $\mathbb{M}$ associated with $M.$ The stretched manifold is a $C^\infty $ manifold with boundary $\partial \mathbb{M}$ which is a locally trivial $X$-bundle over $s_1(M).$ The projection $\mathbb{M}\rightarrow M$ is determined by the bundle projection $\partial \mathbb{M}\rightarrow s_1(M)$ that maps the fibre $X$ over a $y\in s_1(M)$ to $y.$ There is a collar neighbourhood $\mathbb{V}$ of $\partial \mathbb{M}$ in $\mathbb{M}$ which can be identified with an $\overline{\R}_+\times X$-bundle over $s_1(M).$ Then the pointwise projection $\overline{\R}_+\times X\rightarrow X^\Delta =(\overline{\R}_+\times X)/(\{0\}\times X)$ gives us a map $\mathbb{V}\rightarrow V$ where $V$ is a neighbourhood of $s_1(M)$ in $M$ with the structure of an $X^\Delta$-bundle over $s_1(M).$ \\
\nt On a manifold $M$ with edge $s_1(M)$ we consider spaces of operators  
\beq\label{21.E.sesp}
L^\mu (M,{\bf{g}})\subset L_{\textup{cl}}^\mu (s_0(M))\,\,\,\mbox{for}\,\,\,{\bf{g}}=(\gamma ,\gamma -\mu ,\Theta ),
\eeq
$s_0(M)=M\setminus  s_1(M).$ Modulo smoothing operators, to be specified below, the operators are locally near the edge in the splitting of variables $(r,x,y)\in X^\wedge\times \Omega ,\Omega \subseteq \R^q,\,q=\textup{dim}\,s_1(M),$ of the form
\beq\label{21.E.seloc}
r^{-\mu }\Op_{r,y}(p)\,\,\,\mbox{for a}\,\,\,p(r,y,\rho ,\eta ):=\tilde{p}(r,y,r\rho ,r\eta ),
\eeq
\beq\label{21.E.selocp}
\tilde{p}(r,y,\td{\rho} ,\td{\eta} )\in C^\infty (\overline{\R}_+\times \Omega ,L_{\textup{cl}}^\mu (X;\R_{\td{\rho} ,\td{\eta} }^{1+q})).
\eeq
Here $\Op_{r,y}(b)u(r,y):=\iint e^{i(r-r')\rho +i(y-y')\eta }b(r,y,r',y',\rho ,\eta )u(r',y')dr'dy'\dbar\rho \dbar\eta $ for a corresponding operator-valued amplitude function $b,$ and $\dbar\rho :=(2\pi )^{-1}d\rho ,\dbar\eta:=(2\pi )^{-q}d\eta $ (the variable $x$ is suppressed in most cases). In general, for $b(r,y,r',y',\newline\rho ,\eta )\in C^\infty (\R_+\times \Omega \times \R_+\times \Omega,L_{\textup{cl}}^\mu (X;\R^{1+q}))$ we have $\Op_{r,y}(b)\in L^\mu _{\textup{cl}}(\R_+\times X\times \Omega ),$ and any such operator defines a continuous map $C_0^\infty (\R_+\times X\times \Omega )\rightarrow C^\infty (\R_+\times X\times \Omega ).$\\

\nt Let us assume for the moment that $M$ is a compact manifold with edge (the modifications in the non-compact case will be commented afterwards). Let $U_0,\dots,U_N$ be an open covering of the edge $s_1(M)$ by coordinate neighbourhoods, $\{\varphi _j\}_{j=0,\dots,N}$ a subordinate partition of unity, and $\{\varphi '_j\}_{j=0,\dots,N}$ a system of functions $\varphi _j'\in C_0^\infty (U_j),\varphi _j'\succ \varphi _j$ for all $j.$ Moreover, let $\theta \prec \theta  '$ be cut-off functions on the half-axis, and choose $\psi \prec \psi '$ in $C_0^\infty (s_0(M))$ such that $\textup{dist}\,(\textup{supp}\,\psi ,s_1(M))$ is small enough. Then every $A\in L_{\textup{cl}}^\mu (s_0(M))$ can be written in the form 
\beq\label{21.E.sefctglob}
A=\sum_{j=0}^N \theta  \varphi _jB_j\theta '\varphi '_j+\psi A_{\textup{int}}\psi '+C
\eeq
with the following ingredients: $C\in L^{-\infty }(s_0(M)),A_{\textup{int}}\in L_{\textup{cl}}^\mu (s_0(M)),$ and $B_j:=\Op_{r,y}(b_j)$ for amplitude functions $b_j(r,y,\rho ,\eta )\in C^\infty (\R_+\times \Omega ,L^\mu _{\textup{cl}}(X;\R_{\rho ,\eta }^{1+q})).$ Here we use some convenient abbreviations. We identify $U_j$ with an open set $\Omega \subseteq \R^q$ via local coordinates on the edge (for instance, we may take $\Omega =\R^q$), and then, via the charts $U_j\rightarrow \Omega $ we identify $\varphi _j,\varphi _j'$ with functions in $C^\infty _0(\Omega ).$ Moreover, $\Op_{r,y}(b_j)$ is interpreted in combination with the pull back from the local stretched wedge $X^\wedge\times \Omega $ to a corresponding open subset of $s_0(M)$. \\

\nt The operators $A\in L^\mu (M,{\bf{g}})$ will be of a similar structure, namely, $C\in  L^{-\infty } (M,{\bf{g}})$ (the definition will be given below, cf. the formula \eqref{21.E.glsmabb}), $A_{\textup{int}}$ is as in \eqref{21.E.sefctglob}, and 
\beq\label{21.E.sefctb}
B_j:=\Op_y(a_j) \,\,\,\mbox{for an operator-valued amplitude function}\,\,\, a_j(y,\eta )
\eeq
that is of the form 
\beq\label{21.E.sfct}
a_j(y,\eta ):=\epsilon r^{-\mu }\Op_r(p_j)(y,\eta )\epsilon ' \,\,\textup{mod}\, C^\infty (\Omega ,L^{-\infty }(X^\wedge;\R_\eta ^q))
\eeq
for $p_j(r,y,\rho ,\eta )$ as in \eqref{21.E.seloc}, \eqref{21.E.selocp}, moreover, $\epsilon \prec \epsilon ',\theta \prec \epsilon ,\theta '\prec \epsilon '$ are cut-off functions, and the remainder in \eqref{21.E.sfct} is a specific smoothing correction
term coming from edge quantisation, plus contributions that are involved in the control of asymptotics (i.e. smoothing Mellin plus Green symbols). \\

\nt In order to produce $a_j(y,\eta )$ we drop $j$ for a while and apply the edge quantisation, starting from operator functions $r^{-\mu }p(r,y,\rho ,\eta )$ of the form \eqref{21.E.seloc}, \eqref{21.E.selocp}. Here we focus on $r^{-\mu }\Op_r(p)(y,\eta )$ and later on pass to $\Op_y(r^{-\mu }\Op_r(p))=r^{-\mu }\Op_{r,y}(p).$
Next we fix cut-off functions $\omega ''\prec \omega \prec \omega ',$ and set $\chi :=1-\omega ,\chi ':=1-\omega ''.$ For a function $\varphi \in C^\infty (\R_+)$ we write $\varphi _\eta (r):=\varphi (r[\eta ]);$ here $\eta \rightarrow [\eta ]$ is any fixed strictly positive function in $C^\infty (\R^q)$ such that $[\eta ]=|\eta |$ for $|\eta |\geq \varepsilon $ for some $\varepsilon >0.$ We have
\beq\label{21.E.seq1}
r^{-\mu }\Op_r(p)(y,\eta )=r^{-\mu }\omega _\eta \Op_r(p)(y,\eta )\omega '_\eta +r^{-\mu }\chi _\eta \Op_r(p)(y,\eta )\chi '_\eta +c(y,\eta )
\eeq
where $c(y,\eta )=r^{-\mu }\{\omega _\eta\Op_r(p)(y,\eta )(1-\omega '_\eta )+(1-\omega _\eta )\Op_r(p)(y,\eta )\omega ''_\eta \} $ belongs to $C^\infty (\Omega ,L^{-\infty }(X^\wedge;\R_\eta ^q)).$ Since $\Op_y(c)\in L^{-\infty }(X^\wedge\times \Omega )$ we may ignore $c(y,\eta )$ and continue with the first two summands on the right of \eqref{21.E.seq1}.
\begin{thm}\label{21.T.melq0}
Let $p,\td{p}$ be as in \eqref{21.E.seloc}, \eqref{21.E.selocp}; then there is an
\beq\label{21.E.seq1h}
h(r,y,z,\eta ):=\td{h}(r,y,z,r\eta )\,\,\,\mbox{for an}\,\,\,\td{h}(r,y,z,\td{\eta} )\in C^\infty (\overline{\R}_+\times \Omega ,M^\mu _\Ocal(X,\R^q_{\td{\eta}}))
 \eeq
 such that
 \beq\label{21.E.seq12h}
\Op_r(p)(y,\eta )=\op_M^\beta (h)(y,\eta ) \,\,\textup{mod}\,C^\infty(\Omega ,L^{-\infty }(X^\wedge;\R_\eta ^q))
 \eeq
 for every $\beta \in \R.$
\end{thm}
This allows us to write 
\beq\label{21.E.seq2}
r^{-\mu }\Op_r(p)(y,\eta )=r^{-\mu }\omega _\eta \op_M^{\gamma -n/2} (h)(y,\eta ) \omega '_\eta +r^{-\mu }\chi _\eta \Op_r(p)(y,\eta )\chi '_\eta 
\eeq
modulo another remainder in $C^\infty (\Omega ,L^{-\infty }(X^\wedge;\R_\eta ^q))$ that can be dropped again. Now let us choose cut-off functions $\epsilon ''\prec \epsilon \prec \epsilon '.$ Similarly as \eqref{21.E.seq1} we can decompose once again the left hand side of \eqref{21.E.seq2} as
\beq\label{21.E.seq3}
r^{-\mu }\Op_r(p)(y,\eta )=r^{-\mu }\epsilon  \Op_r(p)(y,\eta )\epsilon '+r^{-\mu }(1-\epsilon ) \Op_r(p)(y,\eta )(1-\epsilon '')
\eeq
modulo a smoothing remainder, and then
\begin{equation}\label{21.E.seq4}
r^{-\mu }\Op_r(p)(y,\eta )=a(y,\eta ) + r^{-\mu }(1-\epsilon ) \Op_r(p)(y,\eta )(1-\epsilon '')
\end{equation}
modulo a smoothing remainder, where
\beq\label{21.E.seq5}
a(y,\eta )=r^{-\mu }\epsilon \{\omega _\eta \op_M^{\gamma -n/2} (h)(y,\eta ) \omega '_\eta +\chi _\eta \Op_r(p)(y,\eta )\chi '_\eta\}\epsilon '.
\eeq
The second summand on the right of \eqref{21.E.seq4} is supported far from $r=0,$ and it may be integrated in the interior part of \eqref{21.E.sefctglob}. Therefore, \eqref{21.E.seq5} remains as the specific part of the edge quantisation coming from $r^{-\mu }\Op_r(p)(y,\eta ).$ 
\begin{rem}\label{21.R.seq4}
The modification of $r^{-\mu }\Op_r(p)(y,\eta )$ by a smoothing operator functions close to the edge does not affect the final results concerning parametrices and asymptotics. The precise information will be obtained by adding smoothing Mellin and Green terms that automatically adjust what we lost during the edge quantisation $r^{-\mu }\Op_r(p)(y,\eta )\rightarrow a(y,\eta )$ which is not canonical anyway. 
\end{rem}
\nt The main motivation of edge quantisation is the continuity of operators in weighted edge spaces. First note that when  $H$ is a Hilbert space with group action $\kappa =\{\kappa _\lambda \}_{\lambda \in \R_+}$ (i.e. a strongly continuous group of isomorphisms $\kappa _\lambda :H\rightarrow H$ with $\kappa _\lambda \kappa_ \nu =\kappa _{\lambda \nu} $ for all $\lambda ,\nu \in\R_+$) we have the space $\Wcal^s(\R^q,H),$ defined to be the completion of $\Scal(\R^q,H)$ with respect to the norm $\|\langle\eta \rangle^s\kappa^{-1} _{\langle\eta \rangle}\hat{u} (\eta )\|_{L^2(\R^q,H)} .$ Here $\hat{u} (\eta )=F_{y\rightarrow \eta }u(\eta )$ is the Fourier transform. A similar definition gives us spaces $\Wcal^s(\R^q,E)$ for a Fr\'echet space $E$ with group action. This allows us to form the spaces
\beq\label{21.E.sespac}
\Wcal^s(\R^q,\Kcal^{s,\gamma }(X^\wedge))\,\,\,\mbox{and}\,\,\,\Wcal^s(\R^q,\Kcal_P^{s,\gamma }(X^\wedge))
\eeq
for any asymptotic type $P$ associated with the weight data $(\gamma ,\Theta ).$ Now on a compact manifold $M$ with edge we have the global weighted spaces
\beq\label{21.E.glbspac}
H^{s,\gamma }(M)\,\,\,\mbox{and}\,\,\,H_P^{s,\gamma }(M)
\eeq
which are subspaces of $H^s_{\textup{loc}}(M\setminus s_1(M)),$ locally near $s_1(M)$ modelled on the respective spaces \eqref{21.E.sespac}.\\
\begin{rem}\label{21.R.module} 
The operator of multiplication by a function $\varphi \in C^\infty (\mathbb{M})$ induces a continuous operator \begin{equation}
\varphi :H^{s,\gamma }(M)\rightarrow H^{s,\gamma }(M)\,\,\,\mbox{for every}\,\,\, s,\gamma \in \R, 
\end{equation}
and the corresponding map $C^\infty (\mathbb{M})\rightarrow  \mathcal{L}(H^{s,\gamma }(M))$ is continuous. An analogous result is true of the spaces $H_P^{s,\gamma }(M)$ when the asymptotic type $P$ satisfies the shadow condition.
\end{rem}
\nt The space of smoothing operators of the edge calculus
\beq\label{21.E.glsmoo}
L^{-\infty }(M,{\bf{g}})\,\,\,\mbox{for}\,\,\,{\bf{g}}=(\gamma ,\gamma -\mu ,\Theta )
\eeq
is the set of all $C\in \bigcap _{s\in \R}\Lcal(H^{s,\gamma }(M),H^{\infty ,\gamma -\mu }(M))$ such that
\beq\label{21.E.glsmabb}
C:H^{s,\gamma }(M)\rightarrow H_P^{\infty ,\gamma -\mu }(M),C^*:H^{s,-\gamma +\mu }(M)\rightarrow H_Q^{\infty ,-\gamma }(M)
\eeq
are continuous for every $s\in \R,$ for certain $C$-dependent asymptotic types $P$ and $Q,$ associated with the weight data $(\gamma -\mu ,\Theta )$ and $(-\gamma ,\Theta ),$ respectively. Here $C^*$ is the formal adjoint of $C$ with respect to the $H^{0,0}(M)$-scalar product. Let $L_{\textup{flat}}^{-\infty }(M,{\bf{g}})$ denote the subspace of those $C\in L^{-\infty }(M,{\bf{g}})$ such that 
\beq\label{21.E.glsmaflat}
C:H^{s,\gamma }(M)\rightarrow H_\Theta ^{\infty ,\gamma -\mu }(M),C^*:H^{s,-\gamma +\mu }(M)\rightarrow H_\Theta ^{\infty ,-\gamma }(M)
\eeq
are continuous for every $s\in \R.$
\nt Let us represent the stretched manifold $\mathbb{M}$ associated with $M$ close to $\partial \mathbb{M}$ by a system of \virg{charts} $\mathbb{U}\rightarrow \overline{\R}_+\times X\times \Omega $ (cf. also the formula \eqref{modeg}) where $\mathbb{U}$ runs over a covering of the above-mentioned set $\mathbb{V},$ $\Omega \subseteq \R^q$ open. Without loss of generality we assume that the transition maps $\overline{\R}_+\times X\times \Omega \rightarrow \overline{\R}_+\times X\times \tilde{\Omega }$ are independent of $r$ close to $r=0.$ In other words, near $\partial \mathbb{M}$ we may fix a global axial variable $r\in \overline{\R}_+.$ This allows us the construction of a strictly positive function ${\bf{r}}^1\in C^\infty (\mathbb{M})$ that is equal to $r$ for $0\leq r\leq \varepsilon $ for some $\varepsilon >0$ in the chosen representation of $\mathbb{U}$ in the variables $(r,x,y).$ Let us form the power ${\bf{r}}^\beta $ and call ${\bf{r}}^\beta \in C^\infty (M\setminus s_1(M))$ a weight function with weight $\beta \in \R.$\\
\nt The following definition employs operator-valued symbols referring to group actions between the involved spaces. The generalities will be given in Section 2.4 below.
\begin{defn}\label{21.D.seq4}
A $g(y,\eta )\in \bigcap_{s\in \R} C^\infty (U\times \R^q,\Lcal(\Kcal^{s,\gamma }(X^\wedge),\Kcal^{\infty ,\gamma -\mu }(X^\wedge)))$ for open $U\subseteq \R^p$ is said to be a Green symbol of order $\nu \in \R,$ associated with the weight data ${\bf{g}}:=(\gamma ,\gamma -\mu ,\Theta ),$ if
\beq\label{21.E.segr}
g(y,\eta )\in S^\nu _{\textup{cl}}(U\times \R^q;\Kcal^{s,\gamma ;e}(X^\wedge),\Kcal_P^{\infty ,\gamma -\mu ;\infty }(X^\wedge))
\eeq
and 
\beq\label{21.E.ssgr}
g^*(y,\eta )\in S^\nu _{\textup{cl}}(U\times \R^q;\Kcal^{s,-\gamma +\mu ;e}(X^\wedge),\Kcal_Q^{\infty ,-\gamma  ;\infty }(X^\wedge))
\eeq
for every $s,e\in \R,$ for certain $g$-dependent asymptotic types $P$ and $Q,$ associated with the weight data $(\gamma -\mu ,\Theta )$ and $(-\gamma ,\Theta ),$  respectively, where $g^*$ means the pointwise formal adjoint of $g$ with respect to the $\Kcal^{0,0}(X^\wedge)$-scalar product. Let
\beq\label{21.E.spagr}
R^\mu _G(U\times \R^q,{\bf{g}})_{P,Q}
\eeq
denote the space of all such $g(y,\eta ),$  and set $R^\mu _G(U\times \R^q,{\bf{g}}):=\bigcup _{P,Q}R^\mu _G(U\times \R^q,{\bf{g}})_{P,Q}.$ Moreover, let $R^\nu _G(U\times \R^q)_\Ocal$ be the subspace of those $g(y,\eta )$ such that analogues of the conditions \eqref{21.E.segr} and \eqref{21.E.ssgr} hold when the target spaces on the right are both replaced by $\Kcal^{\infty ,\infty  ;\infty }(X^\wedge)$.
\end{defn} 
\begin{rem}\label{21.R.mu}
For every $\varphi \in C^\infty (\overline{\R}_+\times U)$ which is independent of $r$ for large $r$ and $g \in R^\mu _G(U\times \R^q,{\bf{g}})$ we have $\varphi g,g\varphi \in R^\mu _G(U\times \R^q,{\bf{g}}).$
\end{rem}

\nt Observe that $g(y,\eta )\in R^\nu _G(U\times \R^q)_\Ocal$ is equivalent to the conditions
\beq\label{21.E.segrfl}
g(y,\eta ),g^*(y,\eta )\in S^\nu _{\textup{cl}}(U\times \R^q;\Kcal^{s,\gamma ;e}(X^\wedge),\Kcal^{\infty ,\infty  ;\infty }(X^\wedge))
\eeq
for all $s,\gamma ,e\in \R.$
\begin{rem}\label{21.R.salt}
An alternative to the above-mentioned edge quantisation  $r^{-\mu }\Op_r(p)\newline (y,\eta )\rightarrow a(y,\eta )$ is to set 
\beq\label{21.E.segalg}
a(y,\eta ):=r^{-\mu }\epsilon \op_M^{\gamma -n/2} (h)(y,\eta ) \epsilon ';
\eeq 
more precisely, we have
\beq\label{21.E.segalgtr}
\begin{split}
r^{-\mu }\epsilon \{\omega _\eta \op_M^{\gamma -n/2} (h)(y,\eta ) \omega '_\eta &+\chi _\eta \Op_r(p)(y,\eta )\chi '_\eta\}\epsilon '\\
&-r^{-\mu }\epsilon \op_M^{\gamma -n/2} (h)(y,\eta ) \epsilon '\in R^\mu _G(\Omega \times \R^q)_\Ocal,
\end{split}
\eeq
cf. \cite{Gil2}.
\end{rem}

\nt Another typical ingredient of amplitude functions of the edge calculus are the smoothing Mellin symbols. They are defined in connection with weight data ${\bf{g}}=(\gamma ,\gamma -\mu ,\Theta )$ for a weight interval $\Theta :=(-(k+1),0],k\in \N,$ and have the form
\beq\label{21.E.smels}
m(y,\eta ):=r^{-\mu }\omega _\eta \sum_{j=0}^kr^j\sum_{|\alpha |\leq j}\op_M^{\gamma _{j\alpha }-n/2}(f_{j\alpha })(y)\eta ^\alpha \omega '_\eta 
\eeq
for cut-off functions $\omega ,\omega ',$ Mellin symbols $f_{j\alpha }(y,z)\in C^\infty (\Omega ,M^{-\infty }_{R_{j\alpha }}(X))$ for constant (in $y$) Mellin asymptotic types $R_{j\alpha }$ and weights $\gamma _{j\alpha }$ where $\gamma -j\leq \gamma _{j\alpha }\leq \gamma $ and $\pi _\C R_{j\alpha }\cap \Gamma _{(n+1)/2-\gamma _{j\alpha }}=\emptyset$ for all $j,\alpha .$\\

\nt Recall that
\beq\label{21.E.melsymb}
m(y,\eta )\in S^\mu _{\textup{cl}}(\Omega \times \R^q;\Kcal^{s,\gamma }(X^\wedge),\Kcal^{\infty ,\gamma -\mu }(X^\wedge)), 
\eeq
and
\beq\label{21.E.asmelsymb}
m(y,\eta )\in S^\mu _{\textup{cl}}(\Omega \times \R^q;\Kcal_P^{s,\gamma }(X^\wedge),\Kcal_Q^{\infty ,\gamma -\mu }(X^\wedge)), 
\eeq
for every $s\in \R$ and every asymptotic type $P,$ associated with the weight data $(\gamma ,\Theta )$ for some resulting asymptotic type $Q,$ associated with the weight data $(\gamma -\mu ,\Theta ).$ Moreover, if we change the involved cut-off functions $\omega ,\omega '$ or the $\gamma _{j\alpha }$ (but not the Mellin symbols $f_{j\alpha }$), then we modify $m(y,\eta )$ by a Green symbol for certain asymptotic types. Let $R_{M+G}^\mu (\Omega \times \R^q,{\bf{g}})$ denote the space of all $m(y,\eta )+g(y,\eta )$ for arbitrary $m$ of the form \eqref{21.E.smels} and $g\in R_G^\mu (\Omega \times \R^q,{\bf{g}}).$\\

\nt Finally, let
\beq\label{21.E.sedgess}
R^\mu (\Omega \times \R^q,{\bf{g}})\,\,\,\mbox{for}\,\,\,{\bf{g}}=(\gamma ,\gamma -\mu ,\Theta )
\eeq
denote the set of all operator functions of the form
\beq\label{21.E.sedgeamp}
a(y,\eta ):=a_\Ocal(y,\eta )+(m+g)(y,\eta )
\eeq
for $(m+g)(y,\eta )\in R_{M+G}^\mu (\Omega \times \R^q,{\bf{g}}),$ and $ a_\Ocal(y,\eta ):=r^{-\mu }\epsilon \{\omega _\eta \op_M^{\gamma -n/2} (h)(y,\eta ) \omega '_\eta +\chi _\eta \Op_r(p)(y,\eta )\chi '_\eta\}\epsilon ',$ or, alternatively, $a_\Ocal(y,\eta ):=r^{-\mu }\epsilon \op_M^{\gamma -n/2} (h)(y,\eta )\epsilon '.$\\
\nt Summing up we have spaces of $(y,\eta )$-dependent operator functions on the infinite stretched cone, namely,
$$R^\mu (\Omega \times \R^q,{\bf{g}})\supset R^\mu _{M+G}(\Omega \times \R^q,{\bf{g}})\supset R^\mu _G(\Omega \times \R^q,{\bf{g}}).$$
\begin{defn}\label{21.D.sealg}
The space 
\beq\label{21.E.sealg}
L^\mu (M,{\bf{g}})\,\,\,\mbox{for}\,\,\,{\bf{g}}=(\gamma ,\gamma -\mu ,\Theta ),
\eeq
is defined to be the set of all operators 
\beq\label{21.E.edgectglob}
A=\sum_{j=0}^N \theta  \varphi _j\Op_y(a_j)\theta '\varphi '_j+\psi A_{\textup{int}}\psi '+C
\eeq
for arbitrary $C\in L^{-\infty }(M,{\bf{g}}), A_{\textup{int}}\in L^\mu _{\textup{cl}}(s_0(M)),\psi ,\psi '\in C_0^\infty (s_0(M)),$ and $a_j(y,\eta )\newline\in R^\mu (\Omega \times \R^q,{\bf{g}}).$ Moreover, let $L_{M+G}^\mu (M,{\bf{g}})\,\,(L_G^\mu (M,{\bf{g}}))$ denote the subspace of all elements of \eqref{21.E.sealg} such that $A_{\textup{int}}\equiv 0$ and $a_j\in R_{M+G}^\mu (\Omega \times \R^q,{\bf{g}})\,\,(R_G^\mu (\Omega \times \R^q,{\bf{g}})).$
\end{defn}
\nt For ${\bf{g}}=(\gamma ,\gamma -\mu ,\Theta ), \Theta =(-\infty ,0],$ we set
\beq\label{21.E.edgectgl}
L^\mu (M,{\bf{g}}):=\bigcap _{k\in \N}L^\mu (M,(\gamma ,\gamma -\mu ,(-(k+1) ,0])).
\eeq
Analogously we define $L_{M+G}^\mu (M,{\bf{g}})\,\,(L_G^\mu (M,{\bf{g}}))$ in the case $\Theta =(-\infty ,0]$ as intersections of the corresponding spaces for finite $\Theta .$
\begin{rem}\label{21.R.calg}
\nt In asymptotic considerations it may be useful to distinguish between flat operators of the edge calculus and operators encoding asymptotic properties. Let 
\beq\label{21.E.flat}
L^\mu _{\textup{flat}}(M,{\bf{g}}) \,\,\,\mbox{for}\,\,\, {\bf{g}}=(\gamma ,\gamma -\mu ,\Theta ),
\eeq
$\Theta =(-(k+1) ,0],$ denote the set of all $A\in L^\mu (M,{\bf{g}})$ for which in the representation \eqref{21.E.edgectglob} the involved local smoothing Mellin symbols belong to $ C^\infty (\Omega ,M^{-\infty }_{\mathcal{O}}(X)),$ the local Green symbols to $R^\mu _G(\Omega \times \R^q)_{\Theta ,\Theta },$ and $C\in L^{-\infty }_{\textup{flat}}(M,{\bf{g}}).$ Then we have
\beq\label{21.E.flat+as}
L^\mu (M,{\bf{g}})=L^\mu _{\textup{flat}}(M,{\bf{g}}) +L_{M+G}^\mu (M,{\bf{g}}).
\eeq
Spaces of flat operators for $\Theta =(-\infty ,0]$ are defined again as intersections of corresponding spaces for finite $\Theta ;$ in particular, the local Green symbols then belong to $R^\mu _G(\Omega \times \R^q)_{\mathcal{O}},$ and we have a relation like \eqref{21.E.flat+as} also in the case of infinite $\Theta .$
\nt There is an immediate generalisation of Definition \ref{21.D.sealg} to the case of a non-compact manifold $M$ with edge. The smoothing operators are based on mapping properties of the kind $H^{s,\gamma }_{\textup{[comp)}}(M)\rightarrow H^{\infty ,\gamma -\mu }_{\textup{[loc)},P}(M)$ and $H^{s,-\gamma +\mu }_{\textup{[comp)}}(M)\rightarrow H^{\infty ,-\gamma  }_{\textup{[loc)},Q}(M),$ respectively, while the first sum on the right of \eqref{21.E.edgectglob} is to be replaced by a corresponding locally finite sum referring to an infinite covering $\{U_j\}_{j\in \N}$ of the edge.  
\end{rem}
\begin{rem}\label{21.R.cdffalg}
The space $\textup{Diff}^\mu _{\textup{deg}}(M)$ belongs to \eqref{21.E.edgectgl} for every weight $\gamma \in \R.$
\end{rem}
\begin{rem}\label{21.R.cdsm}
We have $L^\mu (M,{\bf{g}})\cap L^{-\infty }(s_0(M))= L_{M+G}^\mu (M,{\bf{g}}).$
\end{rem}
\nt The operator spaces of Definition \ref{21.D.sealg} have some natural properties in connection with the above-mentioned weight functions. In the following consideration we assume, for simplicity, $\Theta :=(-\infty ,0].$ Let us call an $A\in L^\mu (M,{\bf{g}})$ for ${\bf{g}}=(\gamma ,\gamma -\mu ,\Theta ), $ flat of order $N\in \N$ if ${\bf{r}}^{-N}A\in L^\mu (M,{\bf{g}}).$\\
\nt Observe that the operators $\psi A_{\textup{int}}\psi'$ in \eqref{21.E.edgectglob} are flat of infinite order (i.e. for arbitrary $N$).
\begin{prop}\label{21.P.cdffalg}
Every $A\in L^\mu (M,{\bf{g}})$ can be written in the form
\begin{equation}\label{21.E.cdev}
A=\sum_{l=0}^N{\bf{r}}^lA_l  \,\,\,\textup{mod}\,L_G^\mu (M,{\bf{g}})
\end{equation}
for any $N\in \N,$ for suitable $A_l\in L^\mu (M,{\bf{g}})$ where $A_l$ is flat of order $l$ for every $l.$
\end{prop}
\begin{proof}
By virtue of Definition \ref{21.D.sealg} it suffices to write the local amplitude functions $a_j$ in \eqref{21.E.edgectglob} in the form $a_j=\sum_{l=0}^Nr^la_{jl} + a_{j,(N)}$ modulo $R^\mu _G(\Omega \times \R^q,{\bf{g}}),$ for elements $a_{jl},a_{j,(N)}\in R^\mu (\Omega \times \R^q,{\bf{g}})$ such that $r^{-N}a_{j,(N)}\in R^\mu (\Omega \times \R^q,{\bf{g}}).$ For any fixed $j$ we now drop the subscript and write $a_j(y,\eta )=:a(y,\eta )$ in the form 
$a(y,\eta ):=a_\Ocal(y,\eta )+m(y,\eta )$ modulo $R^\mu _G(\Omega \times \R^q,{\bf{g}}),$ cf. \eqref{21.E.sedgeamp}, where $a_\Ocal(y,\eta ):=r^{-\mu }\epsilon \op_M^{\gamma -n/2} (h)(y,\eta )\epsilon '$ for an $h(r,y,z,\eta )=\tilde{h}(r,y, z,r\eta ),\tilde{h}(r,y, z,\tilde{\eta })\in C^\infty (\overline{\R}_+\times \Omega ,M_{\mathcal{O}}(X;\R^q_{\tilde{\eta }})),$ and $m(y,\eta )$ as in \eqref{21.E.smels}. What concerns $a_\Ocal(y,\eta )$ we simply apply Taylor's formula to $\tilde{h}(r,y, z,\tilde{\eta })$ at $r=0.$ For $m(y,\eta )$ the desired representation is immediate from \eqref{21.E.smels} when $\gamma _{j\alpha }=\gamma $ for all $j,\alpha ,$ namely, $m(y,\eta )=\sum_{l=0}^Nr^lm_l(y,\eta ),$ say, for $N\geq k,$ with
$$m_0(y,\eta ):=r^{-\mu }\omega _\eta \sum_{j=0}^kr^j\sum_{|\alpha |=j}\op_M^{\gamma _{j\alpha }-n/2}(f_{j\alpha })(y)\eta ^\alpha \omega '_\eta ,\dots,$$
$$m_l(y,\eta ):=r^{-\mu }\omega _\eta \sum_{j=l}^kr^{j-l}\sum_{|\alpha |=j-l}\op_M^{\gamma _{j\alpha }-n/2}(f_{j\alpha })(y)\eta ^\alpha \omega '_\eta .$$\\

\nt In the general case we can write $m(y,\eta )=\sum_{l=0}^Nr^l\tilde{m}_l(y,\eta )$, for
$$\tilde{m}_0(y,\eta ):=m_0(y,\eta )+rm_1(y,\eta ),$$
$$\tilde{m}_l(y,\eta ):=r^{-\mu +1}\omega _\eta \sum_{j=l+1}^kr^{j-(l+1)}\sum_{|\alpha |=j-(l+1)}\op_M^{\tilde{\gamma }_{j\alpha }-n/2}(f_{j\alpha })(y)\eta ^\alpha \omega '_\eta ,$$
for $l=1,2,\dots,$ for weights $\tilde{\gamma }_{j\alpha }$ satisfying the conditions $\gamma -1\leq \tilde{\gamma} _{j\alpha }\leq \gamma $ and $\pi _\C R_{j\alpha }\cap \Gamma _{(n+1)/2-\tilde{\gamma} _{j\alpha }}=\emptyset$ for all $|\alpha |=j-(l+1) .$ Here we use the fact that the change of weights which is always possible in the indicated way, gives rise to admitted Green remainders.
\end{proof}
\begin{ex}\label{22.E.dec}
Let us illustrate Proposition \ref{21.P.cdffalg} for $A\in \textup{Diff}^\mu _{\textup{deg}}(M),$ cf. Remark \ref{21.R.cdffalg}. Any such $A$ has the form $A=A_{\textup{mel}}+A_{\textup{flat}}$ where $A_{\textup{flat}}:=\psi A_{\textup{int}}\psi '$ in the notation of \eqref{21.E.edgectglob}, $A_{\textup{int}}\in \textup{Diff}^\mu ( s_0(M)),$ and $A_{\textup{mel}}$ is locally near $s_1(M)$ in the splitting of variables $(r,x,y)\in \R_+\times X\times \Omega $ a Mellin operator $A_{\textup{mel}}=\epsilon  (r)r^{-\mu} \Op_y\op_M^{\gamma -n/2}(h)$ for an $h(r,y,z,\eta )=\tilde{h}(r,y,z,r\eta ),\,\tilde{h}(r,y,z,\tilde{\eta })\in C^\infty (\overline{\R}_+\times \Omega ,M^\mu _{\mathcal{O}}(X;\R^q_{\tilde{\eta }}))$ (here taking values in parameter-dependent differential operators on $X$), and a cut-off function $\epsilon  .$ As noted in the proof of Proposition \ref{21.P.cdffalg} the $l$-th term of \eqref{21.E.cdev} is locally near $r=0$ of the form $\epsilon  (r)r^{-\mu+l} \Op_y\op_M^{\gamma -n/2}(h_l)$ for $h_l(r,y,z,\eta):=\tilde{h}_l(y,z,r\eta),$ determined by the Taylor expansion $\tilde{h}(r,y,z,\tilde{\eta})=\sum_{l\in \N} r^l\tilde{h}_l(y,z,\tilde{\eta})$ at $r=0.$
\end{ex}
\begin{prop}\label{21.P.mod}
$A\in L^\mu (M,{\bf{g}})$ and $\varphi \in C^\infty (\mathbb{M})$ entails $\varphi A,A\varphi \in L^\mu (M,{\bf{g}}).$
\end{prop}
\begin{proof}
From Remark \ref{21.R.module} we see immediately that $L^{-\infty }(M,{\bf{g}})$ is preserved by the multiplication by $\varphi .$ Moreover, the operators $\psi A_{\textup{int}}\psi '$ in \eqref{21.E.edgectglob} obviously survive such multiplications. The amplitude functions $a_j\in R^\mu (\Omega \times \R^q,{\bf{g}})$ can be multiplied from the left by $\varphi $ which is trivial for $a_{\mathcal{O}}.$ Moreover, Theorem \ref{22.T.ad} below on the level of edge amplitide functions shows that $a_{\mathcal{O}}$ also admits the multiplication from the right by $\varphi .$The smoothing Mellin + Green symbols can be multilied from both sides by $\varphi .$ 
 \end{proof}

\subsection{Ellipticity in the edge algebra} 
The ellipticity of an operator $A\in L^\mu (M,{\bf{g}}),{\bf{g}}=(\gamma ,\gamma -\mu ,\Theta ),$
relies on the principal symbolic structure $\sigma (A)=(\sigma _0(A),\sigma _1(A)),$ generalising the one of $A\in \textup{Diff}^\mu _{\textup{deg}}(M),$ mentioned at the very beginning. The interior symbol $\sigma _0(A)$ is nothing else than the homogeneous principal symbol of $A$ as an operator in $L^\mu _{\textup{cl}}(s_0(M)).$ Moreover, locally close to $s_1(M)$ the operator $A$ is equal to \eqref{21.E.seloc} for \eqref{21.E.selocp}, modulo a smoothing operator. Then, $\tilde{\sigma }_0(A):=\tilde{p}_{(\mu )}(r,y,\tilde{\rho },\tilde{\eta  }),$ the parameter-dependent principal symbol of \eqref{21.E.selocp}, is the pseudo-differential analogue of \eqref{i.E.redy}, also called the rescaled symbol of $A.$ Moreover, $\sigma _1(A),$ the homogeneous principal edge symbol of the operator $A,$ is the pseudo-differential generalisation of \eqref{isyeg}, defined as
\beq\label{22.E.edgesy}
\begin{split}
\sigma _1(A)(y,\eta )&:=r^{-\mu }\{\omega _{|\eta |}\op_M^{\gamma -n/2} (h_0)(y,\eta ) \omega '_{|\eta |}+\chi _{|\eta |} \Op_r(p_0)(y,\eta )\chi '_{|\eta |}\}\\
&+r^{-\mu }\omega _{|\eta |} \sum_{j=0}^kr^j\sum_{|\alpha |= j}\op_M^{\gamma _{j\alpha }-n/2}(f_{j\alpha })(y)\eta ^\alpha \omega '_{|\eta |} +g_{(\mu )}(y,\eta ),
\end{split}
\eeq
$(y,\eta )\in T^*(s_1(M))\setminus 0,\,\omega _{|\eta |}(r):=\omega (r|\eta |),$ etc. Here \beq\label{22.E.edgesy1}
p_0(r,y,\rho ,\eta ):=\tilde{p}(0,y,r\rho ,r\eta ),\,h_0(r,y,z,\eta ):=\tilde{h}(0,y,z ,r\eta ).
\eeq
 Moreover, $m+g$ is the smoothing Mellin plus Green part of the local amplitude function of $A,$ cf. \eqref{21.E.sedgeamp}, where $m$ is written in the form \eqref{21.E.smels}, while $g_{(\mu )}$ is the (twisted homogeneous) principal symbol of the Green symbol $g.$ The principal edge symbol defines a family of continuous operators
\beq\label{22.E.princ}
\sigma _1(A)(y,\eta ): \Kcal^{s,\gamma }(X^\wedge)\rightarrow \Kcal^{s-\mu ,\gamma -\mu }(X^\wedge),
\eeq
smoothly depending on $(y,\eta )$ and homogeneous in the sense
\beq\label{22.E.princhom}
\sigma _1(A)(y,\lambda \eta )= \lambda ^\mu \kappa _\lambda \sigma _1(A)(y,\eta ) \kappa _\lambda ^{-1}
\eeq
for all $\lambda \in \R_+,(y,\eta )\in T^*(s_1(M))\setminus 0.$
\begin{defn}\label{22.D.ell}
Let $A\in L^\mu (M,{\bf{g}}),{\bf{g}}=(\gamma ,\gamma -\mu ,\Theta ).$
\begin{itemize}
\item[\textup{(i)}] The operator $A$ is called $\sigma _0$-elliptic, if $\sigma _0(A)$ never vanishes on $T^*(s_0(M))\setminus 0,$ and if in addition close to the edge $\tilde{\sigma }_0(A)(r,x,y,\rho ,\xi ,\eta )\neq 0$ for $(\rho ,\xi ,\eta )\neq 0,$ up to $r=0.$
\item[\textup{(ii)}] The operator $A$ is called $\sigma _1$-elliptic if \eqref{22.E.princhom} is a family of isomorphisms for all $(y,\eta )\in T^*(s_1(M))\setminus 0$ for any $s\in \R.$
\end{itemize}
We call $A$ elliptic if it satisfies the conditions $\textup{(i)},\textup{(ii)}.$
 \end{defn}
 \begin{rem}\label{22.R.ellcon}
The definition of ellipticity admits several generalisations, cf. \cite{Schu32}, namely, in form of bijective families of $2\times 2$ block matrices
\beq\label{22.E.princblock}
\sigma _1(\Acal)(y,\eta ): \Kcal^{s,\gamma }(X^\wedge)\oplus J_{-,y}\rightarrow \Kcal^{s-\mu ,\gamma -\mu }(X^\wedge)\oplus J_{+,y}, \sigma _1(A)=\sigma _1(\mathcal{A})_{11},
\eeq
where $J_{\pm ,y}$ are the fibres of complex vector bundles $J_{\pm }$ over $s_1(M).$ If we write $\sigma _1(\Acal)=(\sigma _1(\Acal_{ij}))_{i,j=1,2},$ then $\sigma _1(\Acal_{21})$ is the symbol of an elliptic edge trace condition and $\sigma _1(\Acal_{12})$ of an elliptic edge potential condition, while $\sigma _1(\Acal_{22})$ is the principal symbol of a classical pseudo-differential operator over $s_1(M)$ (operating between sections of $J_{\pm }$).
The bijectivity of \eqref{22.R.ellcon} is an analogue of the Shapiro-Lopatinskij condition, known from the ellipticity of (pseudo-differential) boundary value problems, cf. \cite{Bout1}, or \cite{Remp2}. In our applications we focus on the case when $J_{\pm }$ are of fibre-dimension $0$ which corresponds to the bijectivity of \eqref{22.E.princ}. In general the bundles $J_{\pm }$ may depend on the weight $\gamma $.  It is an interesting aspect of ellipticity in the edge case to identify those weights where the extra conditions are not necessary. Another generalisation concerns the case when for topological reasons the family \eqref{22.E.princ} does not admit an extension to a bijective block matrix family \eqref{22.E.princblock}. Such a situation is well-known in analogous form in boundary value problems, cf. \cite{Atiy8}. A block matrix pseudo-differential algebra that extends the one of \cite{Bout1} which covers all these situations, including the parametrices in the elliptic case, has been introduced in \cite{Schu37}; the edge analogue in a similar framework has been studied in \cite{Schu42}.
 \end{rem}
\nt Before we start explaining the structure of parametrices of elliptic edge operators we complete the material on the operator spaces of Definition \ref{21.D.sealg}. \\

\nt Let us consider $L^\mu (M,{\bf{g}})$ for ${\bf{g}}=(\gamma ,\gamma -\mu ,\Theta )$ and set
\beq\label{22.E.symbsp}
\sigma \big(L^\mu (M,{\bf{g}})\big):=\{\sigma (A)=(\sigma _0(A),\sigma_1(A)):A \in L^\mu (M,{\bf{g}})\}
\eeq
which is a vector space. The components of $\sigma (A)$ are uniquely determined by $A,$ and hence we have a linear map
\beq\label{22.E.symmap}
\sigma :L^\mu (M,{\bf{g}})\rightarrow \sigma \big(L^\mu (M,{\bf{g}})\big).
\eeq
\begin{prop}
The principal symbolic map \eqref{22.E.symmap} has a right inverse
\beq\label{22.E.symrigh}
\op:\sigma \big(L^\mu (M,{\bf{g}})\big)\rightarrow L^\mu (M,{\bf{g}}),
\eeq
i.e. $\sigma \circ \op=\textup{id}$ on the space $\sigma \big(L^\mu (M,{\bf{g}})\big).$ 
\end{prop}
\begin{proof}
The construction of $\op$ is local in nature. In particular we may argue for symbols or operators over an open set  $V\subset M$ separately for $\textup{dist}\,(V,s_0(M))>0$ and $V\cap s_0(M)=\emptyset.$ The assertion in the first case is well-known; here only $\sigma _0$ is involved. Therefore, we concentrate on a neighbourhood close to $s_1(M),$ represented by $X^\Delta \times \Omega ,$ in the splitting of variables $(r,x,y)\in \Sigma ^\wedge \times \Omega $ where $\Sigma \subseteq \R^n$ corresponds to a chart $U\rightarrow \Sigma  $ on $X.$ The first symbolic component is a function of the form $r^{-\mu }p_{(\mu )}(r,x,y,\rho ,\xi ,\eta )$ with $p_{(\mu )}(r,x,y,\rho ,\xi ,\eta )=\tilde{p}_{(\mu )}(r,x,y,r\rho ,\xi ,r\eta )$ for a $\tilde{p}_{(\mu )}(r,x,y,\tilde{\rho },\xi ,\tilde{\eta })\in C^\infty (\overline{\R}_+\times \Sigma \times \Omega \times (\R^{1+n+q}_{\tilde{\rho },\xi ,\tilde{\eta }}\setminus \{0\})).$ Without loss of generality we may assume that $p_{(\mu )}(r,x,y,\rho ,\xi ,\eta )$ vanishes with respect to $x$ off some compact subset of $\Sigma ;$ the general case can be treated by gluing together the constructions for an open covering of $X$ by coodinate neighbourhoods, using a subordinate partition of unity. For a similar reason it suffices to assume $p_{(\mu )}(r,x,y,\rho ,\xi ,\eta )\equiv 0$ for $r>\varepsilon  $ for some small $\varepsilon >0.$  Let $(a_0,a_1)$ be the given element of $\sigma \big(L^\mu (M,{\bf{g}})\big),$ namely,   
\beq\label{22.E.int}
a_0=r^{-\mu }p_{(\mu )}(r,x,y,\rho ,\xi ,\eta ),
 \eeq
 and
 \beq\label{22.E.edopg}
  a_1=r^{-\mu }\{\omega _{|\eta |}\op_M^{\gamma -n/2} (h'_0)(y,\eta ) \omega '_{|\eta |}+\chi _{|\eta |} \Op_r(p'_0)(y,\eta )\chi '_{|\eta |}\}+c_1
\eeq
for
\beq\label{22.E.edopgc1}
c_1:=r^{-\mu }\omega _{|\eta |} \sum_{j=0}^kr^j\sum_{|\alpha |= j}\op_M^{\gamma _{j\alpha }-n/2}(f_{j\alpha })(y)\eta ^\alpha \omega '_{|\eta |} +g_{(\mu )}(y,\eta ),
\eeq
for certain  $p'_0(r,y,\rho ,\eta )=\tilde{p}'_0(y,r\rho ,r\eta ),\tilde{p}'_0(y,\tilde{\rho },\tilde{\eta })\in C^\infty (\Omega ,L^\mu _{\textup{cl}}(X;\R^{1+q}_{\tilde{\rho },\tilde{\eta }}))$ and $h'_0(r,\newline y,z ,\eta )=\tilde{h}'_0(y, z ,r\eta ),\tilde{h}'_0(y,z,\tilde{\eta })\in C^\infty (\Omega ,M_{\mathcal{O}}^\mu (X;\R^q_{ \tilde{\eta }}))$ where $\Op_r(p'_0)(y,\eta )=\op_M^\beta (h'_0)(y,\eta ) \,\,\textup{mod}\,C^\infty(\Omega ,L^{-\infty }(X^\wedge;\R_\eta ^q)).$ The compatibility condition between $a_0$ and $a_1$ consists of the fact that the parameter-dependent principal symbol $\tilde{p}'_{0,(\mu )}(x,\newline y,\tilde{\rho },\xi ,\tilde{\eta }) $ of $\tilde{p}'_0(y,\tilde{\rho },\tilde{\eta })$ satisfies the relation
\beq\label{22.E.intcomp}
p_{(\mu )}(0,x,y,\rho ,\xi ,\eta )=\tilde{p}'_{0,(\mu )}(x,y,r\rho ,\xi ,r\eta ).
 \eeq 
Now a staightforward modification of a well-known extension theorem of Seeley tells us that there is a $\tilde{p}(r,x,y,\tilde{\rho },\xi ,\tilde{\eta })\in C^\infty (\overline{\R}_+\times \Omega ,L^\mu _{\textup{cl}}(X;\R^{1+q}_{\tilde{\rho },\tilde{\eta }}))$ with the given $\tilde{p}_{(\mu )}(r,x,y,\tilde{\rho },\xi ,\tilde{\eta })$ as its parameter-dependent homogeneous principal symbol, and $\tilde{p}(0,y,\tilde{\rho },\tilde{\eta })=\tilde{p}'_0(y,\tilde{\rho },\tilde{\eta }).$ By virtue of Theorem \ref{21.T.melq0} we can choose an $\td{h}(r,y,z,\td{\eta} )\in C^\infty (\overline{\R}_+\times \Omega ,M^\mu _\Ocal(X,\R^q_{\td{\eta}}))$ such that 
$h(r,y,z,\eta ):=\td{h}(r,y,z,r\eta )$ satisfies the relation $\Op_r(p)(y,\eta )=\op_M^\beta (h)(y,\eta ) \,\,\textup{mod}\,C^\infty(\Omega ,L^{-\infty }(X^\wedge;\R_\eta ^q))$ for every $\beta \in \R.$ It follows that $\td{h}(0,y,z,\td{\eta} )-\tilde{h}'_0(y,z,\tilde{\eta })\in C^\infty (\Omega  ,M^{-\infty } _\Ocal(X,\R^q_{\td{\eta}})).$ Finally with the constructed data we form the operator-valued amplitude function \eqref{21.E.seq5} and set $B:=\Op_y(a).$ Then $\sigma _0(B)=a_0$ and 
\beq\label{22.E.inhi}
\sigma _1(B)=r^{-\mu }\omega _{|\eta |}\op_M^{\gamma -n/2} (h_0)(y,\eta ) \omega '_{|\eta |}+r^{-\mu }\chi _{|\eta |} \Op_r(p_0)(y,\eta )\chi '_{|\eta |}
 \eeq 
 for $h_0(r,y,z,\eta )=\td{h}(0,y,z,r\eta ),p_0(r,y,\rho ,\eta )=\tilde{p}(0,y,r\rho ,r\eta).$ For an excision function $\vartheta (\eta )$ in $\R^q$ we form
\beq\label{22.E.inhiC}
C:=\Op_y(\vartheta (\eta )c_1(y,\eta ))+r^{-\mu }\Op_y(\omega _\eta \op_M^{\gamma -n/2} (h'_0-h_0)(y,\eta ) \omega '_\eta ).
 \eeq
Then for $A:=B+C\in L^\mu (B,{\bf{g}})$ we obtain $\sigma _0(A)=a_0,\sigma _1(A)=a_1,$ and hence we may set $A:=\textup{op}(a_0,a_1).$
\end{proof}

\nt Any choice of \eqref{22.E.symrigh} is called an operator convention of the edge calculus.\\ 

\nt The operators  
$A\in L^{\mu -1}(M,{\bf{g}})\subset L_{\textup{cl}}^{\mu -1}(s_0(M))$ have again a homogeneous principal symbol $\sigma _0^{\mu -1}(A)$ of order $\mu -1$ in the standard sense, and a 
homogeneous principal edge symbol $\sigma _1^{\mu -1}(A)$ of order $\mu -1,$ 
\beq\label{22.E.low1}
 \sigma _1^{\mu -1}(A)(y,\eta ): \Kcal^{s,\gamma }(X^\wedge)\rightarrow \Kcal^{s-\mu +1,\gamma -\mu }(X^\wedge),
 \eeq
parametrised by $(y,\eta )\in T^*(s_1(M)\setminus 0),$ (twisted) homogeneous in the sense $\sigma _1^{\mu -1}(A)(y,\lambda \eta )=\lambda ^{\mu -1}\kappa _\lambda ^{-1}\sigma _1^{\mu -1}(A)(y,\eta )\kappa _\lambda ,\,\lambda \in \R_+.$ Successively we now define the operator spaces $L^{\mu -j}(M,{\bf{g}})\,\,\,\mbox{for}\,\,\,{\bf{g}}=(\gamma ,\gamma -\mu ,\Theta ),j=1,2,\dots,$ by the condition of vanishing $\sigma^{\mu -(j-1)}(A)=( \sigma_0^{\mu -(j-1)}(A),\sigma_1^{\mu -(j-1)}(A) ).$ Let $L_{M+G}^{\mu -j}(M,{\bf{g}}):=L_{M+G}^\mu(M,{\bf{g}})\cap L^{\mu -j}(M,{\bf{g}}),L_G^{\mu -j}(M,{\bf{g}}):=L_G^\mu(M,{\bf{g}})\cap L^{\mu -j}(M,{\bf{g}})$ for any $j\in \N.$\\

\begin{thm}\label{22.T.as}
Let $A_j\in L^{\mu -j}(M,{\bf{g}}),j\in \N, \Theta =(-(k+1),0],$ be an arbitrary sequence, and assume that the asymptotic types involved in the Green operators are independent of $j.$ Then there is an $A\in L^{\mu }(M,{\bf{g}})$ such that $A-\sum_{j=0}^NA_j\in L^{\mu -(N+1)}(M,{\bf{g}})$ for every $N\in \N,$ and $A$ is unique modulo $L^{-\infty }(M,{\bf{g}}).$
\end{thm}
\nt Similarly as in the standard pseudo-differential calculus $A$ is called an asymptotic sum of the $A_j,j\in \N,$ written $A\sim \sum_{j=0}^\infty A_j.$
\begin{rem}\label{22.R.cont}
An operator $A\in  L^m(M,{\bf{g}})$ for ${\bf{g}}=(\gamma ,\gamma -\mu ,\Theta ),\mu -m\in \N,$ induces continuous operators
\beq\label{22.E.contas}
A:H^{s,\gamma }(M)\rightarrow H^{s-m,\gamma -\mu }(M),\,H_P^{s,\gamma }(M)\rightarrow H_Q^{s-m,\gamma -\mu }(M)
\eeq

\nt for every $s\in \R$ and any asymptotic type $P$ associated with $(\gamma ,\Theta ),$ for some resulting $Q$ associated with $(\gamma -\mu ,\Theta ).$
\end{rem}
\begin{thm}\label{22.T.comp}
The composition of operators induces a bilinear map
\beq\label{22.T.comp1}
L^m(M,{\bf{g}})\times L^n(M,{\bf{h}})\rightarrow L^{m+n}(M,{\bf{g}}\circ {\bf{h}})
\eeq
for $\mu -m,\nu -n\in \N,$ and ${\bf{g}}=(\gamma -\nu ,\gamma -(\mu +\nu ),\Theta ),{\bf{h}}=(\gamma ,\gamma -\nu ,\Theta ),{\bf{g}}\circ {\bf{h}}=(\gamma ,\gamma -(\mu +\nu ),\Theta ),$ and we have
\beq\label{22.T.comp2}
\sigma ^{m+n}(AB)=\sigma ^m(A)\sigma ^n(B)
\eeq
with componentwise multiplication. Moreover, if one of the factors belongs to the space with subscript $M+G\,(G),$ then the same is true of the composition.
\end{thm}
\begin{thm}\label{22.T.ad}
Let $A\in L^m(M,{\bf{g}})$ for ${\bf{g}}=(\gamma ,\gamma -\mu ,\Theta ),\mu - m\in \N;$ then for the formal adjoint we have $A^*\in L^m(M,{\bf{g}}^*)$ for ${\bf{g}}^*=(-\gamma +\mu ,-\gamma  ,\Theta )$ \textup{(}defined by $(Au,v)_{H^{0,0}(M)}=(u,A^*v)_{H^{0,0}(M)}$ for all $u,v\in C_0^\infty (s_0(M))$\textup{)}, and 
\beq\label{22.T.ad12}
\sigma ^m(A^*)=\sigma ^m(A)^*
\eeq 
componentwise, i.e. $\sigma _0^m(A^*)=\overline{\sigma _0^m(A)},$ and $\sigma _1^m(A^*)=\sigma _1^m(A)^*$ where \virg{*} on the right means the $(y,\eta )$-wise formal adjoint of $\sigma _1^m(A): \Kcal^{s,\gamma }(X^\wedge)\rightarrow \Kcal^{s-m,\gamma -\mu }(X^\wedge)$ in the sense $(\sigma _1^m(A)f,g)_{\Kcal^{0,0}(X^\wedge)}=(f,\sigma _1^m(A)^*g)_{\Kcal^{0,0}(X^\wedge)}$ for $f,g\in C_0^\infty (X^\wedge).$ Moreover, if $A$ belongs to the space with subscript $M+G\,(G),$ then the same is true of the formal adjoint.
\end{thm}

\subsection{Asymptotic parametrices in the edge case} 
We now turn to the construction of parametrices of elliptic operators. The new aspect here are asymptotic parametrices, i.e. with flat and Green operators as left over terms. Those give rise to a transparent description of asymptotics to elliptic operators in the edge calculus. We first outline the known general ideas of the parametrix construction. Let $A\in  L^\mu (M,{\bf{g}}),\,{\bf{g}}=(\gamma ,\gamma -\mu ,\Theta ).$ The first step is to employ $\sigma _0$-ellipticity of $A$, cf. Definition \ref{22.D.ell}. Apart from the standard ellipticity in $L^\mu _{\textup{cl}}(s_0(M))$ the operator $A$ is locally near $s_1(M)$ of the form $A=r^{-\mu }\Op_{r,y}(p)$ modulo a smoothing operator, where $p$ is as in \eqref{21.E.seloc}, \eqref{21.E.selocp}, and $\tilde{p}(r,y,\tilde{\rho },\tilde{\eta })$
takes values in parameter-dependent elliptic operators in $L^\mu _{\textup{cl}}(X;\R_{\tilde{\rho },\tilde{\eta }}^{1+q})$ up to $r=0.$ 
\begin{lem}\label{23.L.p1}
Under the above-mentioned conditions on $p$ there exists a $p^{(-1)}(r,y,\newline\rho ,\eta )=\tilde{p}^{(-1)}(r,y,r\rho ,r\eta ),\,\tilde{p}^{(-1)}(r,y,\tilde{\rho },\tilde{\eta })\in C^\infty (\overline{\R}_+,L^{-\mu}_{\textup{cl}}(X;\R_{\tilde{\rho },\tilde{\eta }}^{1+q})), $ such that
\beq\label{23.E.p2}
\begin{split}
&r^\mu p^{(-1)}(r,y,\rho ,\eta )\,\sharp _{r,y}\, r^{-\mu }p(r,y,\rho ,\eta )\sim 1,\\
&r^{-\mu }p(r,y,\rho ,\eta )\,\sharp _{r,y}\, r^\mu p^{(-1)}(r,y,\rho ,\eta )\sim 1,
\end{split}
\eeq
and equivalence means equality of corresponding asymptotic sums modulo smoothing families. 
Here, $\sharp _{r,y}$ between operator families $a(r,y,\rho ,\eta ),b(r,y,\rho ,\eta )$ means the Leibniz product in the respective variables, i.e. the asymptotic sum $$\sum_{\alpha \in \N^{1+q}}1/\alpha !\,\partial ^\alpha _{\rho ,\eta }a(r,y,\rho ,\eta )(D^\alpha _{r,y}b(r,y,\rho ,\eta )).$$ 
\end{lem}

\nt Let us first assume that our elliptic operator has the form 
\begin{equation}\label{23.E.p3}
A=\sum_{j=0}^N \vartheta  _j\Op_y(a_j)\vartheta '_j+\psi A_{\textup{int}}\psi '
\end{equation}
for $\vartheta _j:=\theta  \varphi _j\prec \vartheta '_j:=\theta ' \varphi '_j,$ and without loss of generality, $\sum_{j=0}^N \vartheta  _j +\psi \equiv 1,\psi  \prec \psi ',$ and $A_{\textup{int}}=A$  modulo $L^{-\infty }(s_0(M)),$ as an operator in $L^\mu _{\textup{cl}}(s_0(M)).$ Moreover, $a_j(y,\eta )$ is given by
$ r^{-\mu }\epsilon \{\omega _\eta \op_M^{\gamma -n/2} (h_j)(y,\eta ) \omega '_\eta +\chi _\eta \Op_r(p_j)(y,\eta )\chi '_\eta\}\epsilon ',$ or, alternatively, $r^{-\mu }\epsilon \op_M^{\gamma -n/2} (h_j)(y,\eta )\epsilon ',$ i.e., compared with \eqref{21.E.edgectglob}, without smoothing Mellin and Green terms, and we assume $C=0.$ For instance, every $A\in \textup{Diff}^\mu _{\textup{deg}}(M)$ can be written in this way. Clearly $p_j(r,y,\rho ,\eta )=\tilde{p}_j(r,y,r\rho ,r\eta )$ refers to the corresponding local representation of $A$ over $X^\wedge\times \Omega ,$ and $h_j(r,y,z,\eta )=\tilde{h}_j(r,y,z,r\eta )$ is associated with $p_j$ as in Theorem \ref{21.T.melq0}. Let us now fix $j$ and drop it again. \\
\nt Applying Theorem \ref{21.T.melq0} with $p^{(-1)}(r,y,\rho ,\eta )$ of Lemma \ref{23.L.p1} we associate an $h^{(-1)}(r,\newline
y,z ,\eta )=\tilde{h}^{(-1)}(r,y,z ,r\eta ),\td{h}^{(-1)}(r,y,z,\td{\eta} )\in C^\infty (\overline{\R}_+\times \Omega ,M^{-\mu} _\Ocal(X,\R^q_{\td{\eta}})).$ This can be done for every $j,$ and we set
\begin{equation}\label{23.E.p6}
a_j^{-1}(y,\eta ):=r^\mu \epsilon \{\omega _\eta \op_M^{\gamma -\mu -n/2} (h_j^{(-1)})(y,\eta ) \omega '_\eta +\chi _\eta \Op_r(p_j^{(-1)})(y,\eta )\newline\chi '_\eta\}\epsilon ',
\end{equation}
or, alternatively, $a_j^{-1}(y,\eta ):=r^\mu \epsilon \op_M^{\gamma -n/2} (h_j^{(-1)})(y,\eta )\epsilon '.$ Moreover, let $A^{(-1)}_{\textup{int}}\in L^{-\mu }_{\textup{cl}}(s_0(M))$ be any parametrix of the elliptic operator $A_{\textup{int}}$.
\begin{lem}\label{23.L.p4}
The operator
\begin{equation}\label{23.E.p5}
A^{(-1)}:=\sum_{j=0}^N \vartheta  _j\Op_y(a_j^{-1})\vartheta '_j+\psi A^{(-1)}_{\textup{int}}\psi '
\end{equation}
has the properties $A^{(-1)}\in L^{-\mu }(M,{\bf{g}}^{-1})$ for ${\bf{g}}^{-1}=(\gamma -\mu ,\gamma ,\Theta ),$ and
\begin{equation}\label{23.E.p7}
A^{(-1)}A=1\,\,\textup{mod}\,\,L_{M+G}^0(M,{\bf{g}}_l),\,AA^{(-1)}=1\,\,\textup{mod}\,\,L_{M+G}^0(M,{\bf{g}}_r)
\end{equation}
for ${\bf{g}}_l=(\gamma ,\gamma ,\Theta ),{\bf{g}}_r=(\gamma -\mu ,\gamma -\mu ,\Theta ).$
\end{lem}

\begin{proof}
By virtue of the definition of $A^{(-1)}_{\textup{int}}$ and the construction of the sum on the right of \eqref{23.E.p5} we have $A^{(-1)}_{\textup{int}}=A^{(-1)}$  modulo $L^{-\infty }(s_0(M)),$ as an operator in $L^{-\mu }_{\textup{cl}}(s_0(M)).$ It follows that $A^{(-1)}$ is a parametrix of $A$ over $s_0(M)$ in the standard sense. Since both $A:H^{s,\gamma }(M)\rightarrow H^{s-\mu ,\gamma -\mu }(M)$ and $A^{(-1)}:H^{s-\mu ,\gamma -\mu }(M)\rightarrow H^{s,\gamma }(M)$ belong to the edge calculus the compositions $A^{(-1)}A,AA^{(-1)}$ are well-defined, and we have $A^{(-1)}A-1\in L^0(M,{\bf{g}}_l)\cap L^{-\infty }(s_0(M)),$ and $AA^{(-1)}-1\in L^0(M,{\bf{g}}_r)\cap L^{-\infty }(s_0(M)).$ In order to complete the proof it suffices to apply Remark \ref{21.R.cdsm}.
\end{proof}
\begin{rem}\label{23.R.flat}
Observe that the smoothing Mellin plus Green remainders in the relations \eqref{23.E.p7} are flat in the sense that locally the occurring smoothing Mellin symbols belong to $C^\infty (\Omega ,M_{\mathcal{O}}^{-\infty }(X))$ and the Green symbols to $R^0_G(\Omega \times \R^q)_{\mathcal{O}},$ cf. the notation in Definition \ref{21.D.seq4}.
\end{rem}
\begin{rem}\label{23.R.pellco}
Let $A$ be elliptic with respect to $\sigma _0(A),$ cf. Definition \ref{22.D.ell}. Then $\sigma _{\textup{c}}(A)(y,z):=h(0,y,z,0)$ is a holomorphic family of elliptic operators in $L_{\textup{cl}}^{\mu }(X),$ smoothly depending on $y.$ For every $y$ there exists a countable set $D(y)\subset \C$ with $D(y)\cap \{c\leq \textup{Re}\,z\leq c'\}$ finite for every $c\leq c',$ such that 
\begin{equation}\label{23.E.bij}
\sigma _{\textup{c}}(A)(y,z):H^s(X)\rightarrow H^{s-\mu }(X)
\end{equation}
is bijective for all $z\in \C\setminus D(y)$ and all $s.$ There is a Mellin asymptotic type $R(y)$ such that $\sigma _{\textup{c}}(A)(y,z)^{-1} $ extends from $\C\setminus D(y)$ to an element of $M_{R(y)}^{-\mu }(X).$
\end{rem}
\begin{prop}\label{23.P.pell}
Let $A$ be elliptic with respect to $\sigma _0(A),$ cf. Definition \textup{\ref{22.D.ell}}, assume that the asymptotic type $R(y)$ of Remark \textup{\ref{23.R.pellco}} does not depend on $y,$ and let \eqref{23.E.bij} be bijective for all $z\in \Gamma _{(n+1)/2-\gamma }$ and all $y$ \textup{(}i.e. $\pi _\C R\cap \Gamma _{(n+1)/2-\gamma }=\emptyset$\textup{)}. Then there is an operator $B\in L_{M+G}^{-\mu }(M,{\bf{g}}^{-1})$ such that for $Q:=A^{(-1)}+B$ we have
\begin{equation}\label{23.E.p9}
QA=1\,\,\textup{mod}\,\,L_G^0(M,{\bf{g}}_l)+L^{-1}_{M+G}(M,{\bf{g}}_l),\,AQ=1\,\,\textup{mod}\,\,L_G^0(M,{\bf{g}}_r)+L^{-1}_{M+G}(M,{\bf{g}}_r)
\end{equation}
for ${\bf{g}}_l=(\gamma ,\gamma ,\Theta ),{\bf{g}}_r=(\gamma -\mu ,\gamma -\mu ,\Theta ).$
\end{prop}
\begin{proof}
The bijectivity of $\sigma _1(A)(y,\eta )$ as a family of operators \eqref{22.E.princ} is an ellipticity condition of the respective cone operators over $X^\wedge$ at $r=0$ which contains the bijectivity of the principal conormal symbol
\begin{equation}\label{23.E.conp9}
\sigma _{\textup{c}}(A)(y,z)=\tilde{h}(0,y,z,0)=h(0,y,z,0):H^s(X)\rightarrow H^{s-\mu }(X)
\end{equation}
for all $z\in \Gamma _{(n+1)/2-\gamma }$ and all $y.$ From Theorem \ref{22.T.comp} and Lemma \ref{23.L.p4} we conclude
$\sigma _1(A^{(-1)}A)(y,\eta )=\sigma _1(A^{(-1)})(y,\eta )\sigma _1(A)(y,\eta )$ and $T^\mu (\sigma _{\textup{c}}(A^{(-1)})(y,z))\sigma _{\textup{c}}(A)(y,z)\newline=1+l(y,z)$ for some $l(y,z)\in C^\infty (\Omega ,M_{\mathcal{O}}^{-\infty }(X)).$ In the notation of \eqref{23.E.p6} (for omitted $j$) we have $\sigma _{\textup{c}}(A^{(-1)})(y,z )=h^{(-1)}(0,y,z,0).$ This gives us the relation
\begin{equation}\label{23.E.conp10}
h^{(-1)}(0,y,z+\mu ,0)h(0,y,z,0)=1+l(y,z)\,\,\, \mbox{for all}\,\,\,z\in \C.
\end{equation}
For analogous reasons we have
\begin{equation}\label{23.E.conp11}
h(0,y,z-\mu ,0)h^{(-1)}(0,y,z,0)=1+m(y,z)\,\,\, \mbox{for all}\,\,\,z\in \C
\end{equation}
for some $m(y,z)\in C^\infty (\Omega ,M_{\mathcal{O}}^{-\infty }(X)).$
For the moment we now fix $y$  
and show that there is a $g(y,z)\in M_R^{-\infty }(X)$ for some Mellin asymptotic type $R$ such that
\begin{equation}\label{23.E.conp13}
(h^{(-1)}(0,y,z+\mu ,0)+g(y,z+\mu))h(0,y,z,0)=1\,\,\, \mbox{for all}\,\,\,z\in \C.
\end{equation}
To this end we write for abbreviation $h_\mu ^{(-1)}:=h^{(-1)}(0,y,z+\mu ,0),g_\mu :=g(y,z+\mu),h:=h(0,y,z,0).$ Then the ansatz $(h_\mu ^{(-1)}+g_\mu)h=1$ to determine $g_\mu$ yields $(h_\mu ^{(-1)}+g_\mu)hh_\mu ^{(-1)}=h_\mu ^{(-1)}.$
Using \eqref{23.E.conp11} it follows that $(h_\mu ^{(-1)}+g_\mu)(1+m_\mu )=h_\mu ^{(-1)}$ for $m_\mu:=m(y,z+\mu ).$ Thus, by virtue of Theorems \ref{11.T.mmu} and \ref{11.T.minv}, $g_\mu:=h_\mu ^{(-1)}(1+m_\mu )^{-1}-h_\mu ^{(-1)}$ is as desired. In other words we have $$\sigma _{\textup{c}}(A)(y,z)^{-1} =h^{(-1)}(0,y,z+\mu ,0)+g(y,z+\mu),$$
and by assumption we can choose the decomposition on the right of the latter equation in such a way that $g(y,z)\in C^\infty (\Omega ,M_R^{-\infty }(X))$ for some $y$-independent Mellin asymptotic type $R.$ To complete the proof it 
suffices to set
$$B=r^\mu \Op_y\omega _\eta \op_M^{\gamma -\mu -n/2}(g)(y)\omega '_\eta .$$
In fact, the first relation of \eqref{23.E.p9} holds when $\sigma _{\textup{c}}(QA)=T^\mu (\sigma _{\textup{c}}(Q))\sigma _{\textup{c}}(A)=1.$ However, this is true by construction, since $\sigma _{\textup{c}}(Q)=h^{(-1)}(0,y,z ,0)+g(y,z)$ and $\sigma _{\textup{c}}(A)=h(0,y,z,0),$ cf. the relation \eqref{23.E.conp13}. The second relation of \eqref{23.E.p9} can be verified in an analogous manner.
\end{proof}

\begin{cor}\label{23.C.pit}
Let $A\in L^\mu (M,{\bf{g}}),{\bf{g}}=(\gamma ,\gamma -\mu ,\Theta ),$ be as in Proposition \textup{\ref{23.P.pell}}. Then the operator $Q\in L^{-\mu} (M,{\bf{g}}^{-1}),{\bf{g}}^{-1}=(\gamma -\mu ,\gamma  ,\Theta ),$ can be found in such a way that
\begin{equation}\label{23.E.p9Q}
QA=1\,\,\textup{mod}\,\,L_G^0(M,{\bf{g}}_l),\,AQ=1\,\,\textup{mod}\,\,L_G^0(M,{\bf{g}}_r).
\end{equation}
\end{cor}

\nt In fact, if we denote by $M+G\in L_G^0(M,{\bf{g}}_l)+L^{-1}_{M+G}(M,{\bf{g}}_l)$ the remainder in the first relation of \eqref{23.E.p9}, then we have $(\sum_{j=0}^k(-1)^j(M+G)^j)(1+M+G)=1 \,\,\textup{mod}\,\,L_G^0(M,{\bf{g}}_l);$ here $\Theta =(-(k+1),0].$ Thus, if first $Q$ is as in \eqref{23.E.p9}, the composition $Q_l:=(\sum_{j=0}^k(-1)^j(M+G)^j)Q$ satisfies the relation $Q_lA=1\textup{mod}\,\,L_G^0(M,{\bf{g}}_l).$ In a similar manner we find a $Q_r$ such that $AQ_r=1\,\,\textup{mod}\,\,L_G^0(M,{\bf{g}}_r).$ A standard algebraic argument shows that $Q_l$ or, alternatively, $Q_r$ may be taken as the new operator $Q$ to satisfy the relations \eqref{23.E.p9Q}.
\begin{rem}\label{23.R.pellcut}
Observe that $g(y,z)\in C^\infty (\Omega ,M_R^{-\infty }(X))$ can be directly deduced from $h^{-1}(0,y,z,0)$ which is by assumption an element of $C^\infty (\Omega ,M_R^{-\mu  }(X)),$ cf. Remark \ref{11.R.mdec}.
\end{rem}
\nt  We did not employ so far the condition \textup{(ii)} of Definition \ref{22.D.ell}. This will be the next point.
Observe that the ellipticity of $A$ which includes the bijectivity of \eqref{22.E.princ} for all $(y,\eta ),\eta \neq 0,$ entails the bijectivity of \eqref{23.E.bij} for all $z\in \Gamma _{(n+1)/2-\gamma }$ and all $y.$ In fact, the  bijectivity of
$$\sigma _1(A)(y,\eta ):\Kcal^{s,\gamma }(X^\wedge)\rightarrow \Kcal^{s-\mu ,\gamma -\mu }(X^\wedge)$$ 
means that $\sigma _1(A)(y,\eta )$ as a Fredholm operator between the Kegel-spaces is ell

iptic in the cone algebra over $X^\wedge$ for every fixed $y$ and $\eta \neq 0,$ and the respective bijectivity of $\sigma _{\textup{c}}(A)(y,z)$ belongs to these ellipticity conditions. 
\begin{thm}\label{23.T.param}
Let $A\in L^\mu (M,{\bf{g}})$ for ${\bf{g}}=(\gamma ,\gamma -\mu ,\Theta )$ be elliptic in the sense of Definition \textup{\ref{22.D.ell}}, and assume that the asymptotic type $R(y)$ of Remark \textup{\ref{23.R.pellco}} does not depend on $y$. Then there exists a parametrix $P\in L^{-\mu }(M,{\bf{g}}^{-1})$ for ${\bf{g}}^{-1}=(\gamma -\mu ,\gamma ,\Theta ),$ i.e. $PA=1\,\,\textup{mod}\,\,L^{-\infty }(M,{\bf{g}}_l),AP=1\,\,\textup{mod}\,\,L^{-\infty }(M,{\bf{g}}_r)$ for ${\bf{g}}_l,{\bf{g}}_r$ as in Proposition \textup{\ref{23.P.pell}}.
\end{thm}
\begin{proof}
Let us first assume that $A$ is of the special form \eqref{23.E.p3} , and let $Q$ be as in Corollary \ref{23.C.pit}. The edge symbol $\sigma _1(Q)(y,\eta ):\Kcal^{s-\mu ,\gamma -\mu }(X^\wedge)\rightarrow \Kcal^{s,\gamma }(X^\wedge)$ is of index $0.$ In fact, from \eqref{23.E.p9Q} we know that $\sigma _1(Q)(y,\eta )\sigma _1(A)(y,\eta )$ as well as $\sigma _1(A)(y,\eta )\sigma _1(Q)(y,\eta )$ have the form $1+\sigma _1(G)(y,\eta )$ for some $G\in L_G^0(M,{\bf{g}}_l)$ (or $L_G^0(M,{\bf{g}}_r)).$ Since Green symbols are compact in the Kegel-spaces, it follows that $\sigma _1(Q)(y,\eta )$ takes values in Fredholm operators of index $0.$ We will construct $P$ by means of the inverse $\sigma _1(P)(y,\eta ):=\sigma _1(A)(y,\eta )^{-1}$ which will be found in the form $\sigma _1(P)(y,\eta )=\sigma _1(Q)(y,\eta )+\sigma _1(G_1)(y,\eta )$ for a suitable $G_1\in L_G^{-\mu }(M,{\bf{g}}^{-1}).$ This is all what we need, since we already have $\sigma _0(P):=\sigma _0(Q)=\sigma _0(A)^{-1}.$ A final step then allows us to construct $P$ by a formal Neumann series argument, here in the following version. From $\sigma (P)=(\sigma _0(P),\sigma _1(P))$ we pass to an operator $P_0\in L^{-\mu  }(M,{\bf{g}}_l)$ such that $\sigma (P_0)=\sigma (P).$ Then $P_0A=1+C$  for a $C\in L^{-1}(M,{\bf{g}}_l)\cap L^{-\infty }(s_0(M))=L^{-1}_{M+G}(M,{\bf{g}}_l)$ since $\sigma (P_0A-1)=\sigma (P_0)\sigma (A)-1=0,$ cf. the notation \eqref{22.E.low1}. The asymptotic sum $\sum_{j=0}^\infty (-1)^jC^j$ exists in the form $1+D,D\in L^{-1  }(M,{\bf{g}}_l)$ (the assumption of Theorem \ref{22.T.as} on the sequence $(-1)^jC^j,j\in \N,$ is satisfied). A similar construction is possible with compositions in the opposite order. Thus $P:=(1+D)P_0$ is as desired. So it remains to construct the above-mentioned $\sigma _1(G_1)(y,\eta ).$ First we carry out the construction for all $\{(y,\eta)\in \overline{U}(y_0)\times S^{q-1}:|\eta -\eta _0|<\varepsilon \}$ for a sufficiently small neighbourhood $U(y_0)$ of a point $y_0$ and sufficiently small $\varepsilon >0.$ Since
\beq\label{23.E.ind}
\sigma _1(Q)(y,\eta ):\Kcal^{s-\mu ,\gamma -\mu }(X^\wedge)\rightarrow \Kcal^{s ,\gamma  }(X^\wedge)
\eeq
is of index $0$ we can fill up \eqref{23.E.ind} to a $2\times 2$ block matrix family of bijective operators
\beq\label{23.E.bi}
f(y,\eta )=(f_{ij}(y,\eta ))_{i,j=1,2}:\Kcal^{s-\mu ,\gamma -\mu }(X^\wedge)\oplus \C^N\rightarrow \Kcal^{s ,\gamma  }(X^\wedge)\oplus \C^N
\eeq
for a suitable $N$ and $f_{11}:=\sigma _1(Q),$ smooth in $(y,\eta )\in \{ U(y_0)\times S^{q-1}:|\eta -\eta _0|<\varepsilon \}$ where $f_{12}$ takes values in $\mathcal{L}(\C^N,C_0^\infty (X^\wedge))$ and $f_{21}=(f_{21,1},\dots,f_{21,N})$ is a vector of maps $f_{21,k}:\Kcal^{s-\mu ,\gamma -\mu }(X^\wedge)\rightarrow \C$ defined by $f_{21,k}u:=\int_{X^\wedge}f_{21,k}(r,x)u(r,x)drdx$ for certain $(y,\eta )$-dependent functions $f_{21,k}(r,x)\in C_0^\infty (X^\wedge),\,k=1,\dots,N.$ The right lower corner $f_{22}$ is simply a smooth function in $(y,\eta )$ with values in $N\times N$
matrices. Generalities on constructions of that kind may be found in \cite[Section 3.3.4]{Haru13}. Here we use, in particular, the compactness of the closure of $V_1:=\{(y,\eta)\in U(y_0)\times S^{q-1}:|\eta -\eta _0|<\varepsilon \}.$ Since the invertible $N\times N$ matrices form an open dense subset of $\C^{N^2}$ a small perturbation of $f_{22}(y_0,\eta _0)$ allows us to pass to another $f(y,\eta )$ for which $f_{22}(y_0,\eta _0)$ is invertible. In other words, in the choice of \eqref{23.E.bi} we may assume the invertibility of $f_{22}(y_0,\eta _0).$ Then $f_{22}(y,\eta )$ is automatically invertible for all $(y,\eta )$ in a neighbourhood of $(y_0,\eta _0).$ Thus by choosing the diameter of $V_1$ sufficiently small we have the invertibility of the right lower corner for all $(y,\eta )$ in that set. This allows us to pass to a modified upper left corner, namely, 
\beq\label{23.E.bmod}
\tilde{f}_{11}(y,\eta ):=f_{11}(y,\eta )+g_{11}(y,\eta ):\Kcal^{s-\mu ,\gamma -\mu }(X^\wedge)\rightarrow \Kcal^{s ,\gamma  }(X^\wedge)
\eeq
for $g_{11}(y,\eta):=-(f_{12}f_{22}^{-1}f_{21})(y,\eta )$ which is invertible in a small neighbourhood of $(y_0,\eta _0).$ Now let us extend $g_{11}(y,\eta)$ by homogeneity to the set $V:=\{(y,\eta )\in U(y_0)\times (\R^q\setminus \{0\}):(y,\eta /|\eta |)\in V_1\}$ by setting
\beq\label{23.E.ext}
g_{11}(y,\eta):=|\eta |^{-\mu }\kappa _{|\eta |}g_{11}(y,\eta /|\eta |)\kappa ^{-1}_{|\eta |}
\eeq
By construction $g_{11}(y,\eta)$ is the homogeneous principal part of order $-\mu $ of a flat Green symbol in the conical set $V.$ Let us now form the composition 
\beq\label{23.E.eco}
(\sigma _1(Q)+g_{11})(y,\eta)\sigma _1(A)(y,\eta )=(1+\sigma _1(G)+f_0)(y,\eta ):\Kcal^{s ,\gamma  }(X^\wedge)\rightarrow \Kcal^{s ,\gamma  }(X^\wedge)
\eeq
for $f_0:=g_{11}\sigma _1(A),(y,\eta)\in V$ (recall that $\sigma _1(Q)=f_{11}$). The operators \eqref{23.E.eco} are invertible for all $(y,\eta)\in V,$ and $\eqref{23.E.eco}$ has the form $1+g_0(y,\eta)$ for the homogeneous principal part $g_0(y,\eta)=(\sigma _1(G)+f_0)(y,\eta )$ of order $0$ of a corresponding Green symbol. Since $f_0$ maps to flat functions, the asymptotics of the functions in the image of $g_0$ and $\sigma _1(G)$ coincide. We have $(1+g_0(y,\eta))^{-1}=1+g_1(y,\eta)$ for the homogeneous principal part $g_1$ of order $0$ of another Green symbol with the same asymptotics in the image as for $\sigma _1(G)$. In fact, $(1+g_0)(1+g_1)=1$ gives us $g_1=-g_0(1+g_1).$ We obtain
\beq\label{23.E.ecinv}
(1+g_1(y,\eta))(\sigma _1(Q)+g_{11})(y,\eta)\sigma _1(A)(y,\eta )=1,
\eeq
i.e. $(1+g_1(y,\eta))(\sigma _1(Q)+g_{11})(y,\eta)$ is a left inverse of $\sigma _1(A)(y,\eta )$ in the set $V.$ In a similar manner we can construct a right inverse, and hence
\beq\label{23.E.ecinv1}
\sigma _1(A)^{-1}(y,\eta )=(1+g_1(y,\eta))(\sigma _1(Q)+g_{11})(y,\eta).
\eeq
This is true first in $V;$ however, the above-mentioned point $(y_0,\eta _0)$ is arbitrary. Therefore, we computed the inverse of $\sigma _1(A)(y,\eta )$ everywhere. All compositions are controlled within the algebra of edge symbols. In our case we have $\sigma _1(G_1)(y,\eta ):=g_1(y,\eta)\sigma _1(Q)(y,\eta)+g_1(y,\eta)g_{11}(y,\eta)+g_{11}(y,\eta)$ which is of Green type.\\
\nt Let us now assume that $A\in L^\mu (M,{\bf{g}})$ also contains a non-trivial Mellin plus Green summand, i.e. $A$ can be written as
\beq\label{23.E.ensm}
A=A_0+M+G+C
\eeq
where $A_0$ is of the form \eqref{23.E.p3}. Clearly we may assume $C=0.$ We first show an analogue of Proposition \ref{23.P.pell}; then the rest of the proof is exactly as before. In the present case instead of \eqref{23.E.conp9} we have
\beq\label{23.E.ecocno}
\sigma _{\textup{c}}(A)(y,z)=h(0,y,z,0)+\sigma _{\textup{c}}(M)(y,z)
\eeq
where $\sigma _{\textup{c}}(M)(y,z)=f_{00}(y,z)$ when $M$ is given in the form \eqref{21.E.smels}. It suffices to show the existence of a $g(y,z)\in C^\infty (\Omega ,M_R^{-\infty }(X))$ such that
$$(h^{-1}(0,y,z+\mu ,0)+g(y,z+\mu ))(h(0,y,z,0)+f_{00}(y,z))=1.$$
Let us write again $h_\mu :=h^{(-1)}(0,y,z+\mu ,0),g_\mu :=g(0,z+\mu ),h:=h(0,y,z,0)$ and make the ansatz $(h_\mu ^{(-1)}+g_\mu)(h+f_{00})=1$ to determine $g_\mu.$ Applying Theorem \ref{11.T.minverse} we see that ($y$-wise) $(h+f_{00})^{-1}=h_\mu ^{(-1)}+r$ for some $r\in M_R^{-\infty }(X).$ This gives us $h_\mu ^{(-1)}+g_\mu =h_\mu ^{(-1)}+r$ and hence $r=g_\mu.$ Finally if $A^{(-1)}$ is associated with $A_0$ as in Proposition \ref{23.P.pell} we find a $B\in L^{-\mu }_{M+G}(M,{\bf{g}}^{-1})$ such that $Q:=A^{(-1)}+B$ has the properties \eqref{23.E.p9}, first from the left, and then, analogously, from the right. It follow a refined $Q$ as in Corollary \ref{23.C.pit} such that the relations \eqref{23.E.p9Q} hold. Then we proceed as in the first part of the proof and obtain a two-sided parametrix $P$ of \eqref{23.E.ensm}.
\end{proof}
\begin{rem}\label{23.R.spa}
With $P\in L^{-\mu }(M,{\bf{g}}^{-1})$ also $P+C$ for arbitrary $C\in L^{-\infty }(M,{\bf{g}}^{-1})$ is a parametrix of $A\in L^\mu (M,{\bf{g}});$ moreover, if $P,Q\in L^{-\mu }(M,{\bf{g}}^{-1})$ are parametrices of $A,$ then $P-Q\in L^{-\infty }(M,{\bf{g}}^{-1}).$
\end{rem}
\begin{thm}\label{23.T.as}
Let $A\in L^\mu (M,{\bf{g}})$ be as in Theorem \textup{\ref{23.T.param}}. Then $Au=f$ for $u\in H^{-\infty ,\gamma }(M)$ and  $f\in H^{s-\mu  ,\gamma -\mu }(M),s\in \R,$ entails $u\in H^{s ,\gamma }(M).$ If in addition $f\in H_S^{s-\mu  ,\gamma -\mu }(M)$ for an asymptotic type $S$ associated with $(\gamma -\mu ,\Theta )$ it follows that $u\in H_R^{s ,\gamma }(M)$ for a resulting asymptotic type $R$ associated with $(\gamma ,\Theta ).$
\end{thm}
\begin{proof}
Let $P\in L^{-\mu }(M,{\bf{g}}^{-1})$ a parametrix of $A.$ Then $Au=f$ implies $PAu=(1+C)u=Pf,$ i.e. $u=Pf-Cu,$ and it suffices to apply \eqref{22.E.contas}.
\end{proof}
\begin{rem}\label{23.R.eig}
The resulting asymptotic type in Theorem \ref{23.T.as} is not affected by the choice of the parametrix $P.$ In fact, if $P+D,D\in L^{-\infty }(M,{\bf{g}}^{-1}),$ is another parametrix of $A,$ cf. Remark \ref{23.R.spa}, then $(P+D)Au=(1+C)u+DAu=(P+D)f$ gives us $u=(P+D)f-DAu-Cu=Pf-Cu.$\\
\end{rem}

\nt Let us now turn to asymptotic parametrices. The general idea is similar as in the conical case. Let $A\in L^\mu (M,{\bf{g}})$ be as in Theorem \ref{23.T.param}, and let $P\in L^{-\mu }(M,{\bf{g}}^{-1})$ be a parameterix. Then applying Proposition \ref{21.P.cdffalg} to $P$ we obtain an expansion
\beq\label{23.E.aspa}
P=\sum_{m=0}^N{\bf{r}}^mP_m  \,\,\,\textup{mod}\,L_G^{-\mu }(M,{\bf{g}}^{-1})
\eeq
for every $N\in \N.$ Let us call an expansion \eqref{23.E.aspa} an asymptotic parametrix of $A.$ Writing $A$ itself in the form \eqref{21.E.cdev} the operator $A_0$ is elliptic in $L^\mu (M,{\bf{g}}),$ and $P_0\in L^{-\mu }(M,{\bf{g}}^{-1})$ in \eqref{23.E.aspa} can be taken as the parametrix of $A_0,$  obtained by Theorem \ref{23.T.param}. The construction of $P_0$ is easier than that of $P$ itself because $P_0$ concerns the case of constant coefficients with respect to $r.$ Now the idea of asymptotic parametrices is to obtain \eqref{23.E.aspa} by an iterative process, starting with the ansatz
\beq\label{23.E.aspit}
\big(\sum_{m=0}^N{\bf{r}}^mP_m \big)\big(\sum_{i=0}^N{\bf{r}}^iA_i \big)\sim 1 ,\,\big(\sum_{i=0}^N{\bf{r}}^iA_i \big)\big(\sum_{m=0}^N{\bf{r}}^mP_m \big)\sim 1 ;
\eeq
$\sim $ means equality modulo a Green plus flat remainder (the remainders will be characterised below). Let us discuss the first relation of \eqref{23.E.aspit}; the second one is similar. The operator $P_0$ is already the parametrix of $A_0.$ In the construction we first replace ${\bf{r}}^mP_m$ by $\boldsymbol{P_m},$ where $\boldsymbol{P_0}:=P_0.$ For $\boldsymbol{P_1}$ we set
\beq\label{23.E.asp1}
\boldsymbol{P_1}:=-P_0{\bf{r}}A_1P_0.
\eeq
This operator satisfies the relation
\beq\label{23.E.asp11}
\boldsymbol{P_1}A_0+P_0{\bf{r}}A_1=-P_0{\bf{r}}A_1C_l\in L^{-\infty }(M,{\bf{g}}_l)
\eeq
where $C_l:=P_0A_0-1\in L^{-\infty }(M,{\bf{g}}_l).$ In general, for $j>0,$ motivated by the desirable identity
$\sum_{m+i=j}\boldsymbol{P_m}{\bf{r}}^iA_i=0,\,\,\,\mbox{or}\,\,\,\boldsymbol{P_j}A_0=-\sum_{m+i=j,m<j}{\bf{r}}^mP_m{\bf{r}}^iA_i,$ we set
\beq\label{23.E.asp1mj}
\boldsymbol{P_j}:=-\big(\sum_{m+i=j,m<j}\boldsymbol{P_m}{\bf{r}}^iA_i\big)P_0,
\eeq
and obtain
\beq\label{23.E.asp1mj3}
\sum_{m+i=j}\boldsymbol{P_m}{\bf{r}}^iA_i=-\big(\sum_{m+i=j,m<j}\boldsymbol{P_m}{\bf{r}}^iA_i\big)C_l\in L^{-\infty }(M,{\bf{g}}_l)
\eeq 
for any $j.$ Also in the edge case the operators \eqref{23.E.asp1mj} do not automatically coincide with those in \eqref{23.E.aspa}.The operators \eqref{23.E.asp1mj} can be expressed in terms of $P_0$ and ${\bf{r}}^iA_i,$ for $i=1,\dots,j,$ alone. For instance, we have anlogues of the formulas \eqref{13.E.aspber2}, \eqref{13.E.aspber3}, 
etc. Also the other elements of the asymptotic parametrix construction of Section 1.3 can be carried out for the edge case. In particular we obtain the following result.
\begin{thm}\label{23.P.as}
Let $A\in L^\mu (M,{\bf{g}})$ be an elliptic operator. Then for any fixed $N\in \N,N\geq 1,$ the asymptotic parametrix $P_{\textup{as}}$ can be chosen in such a way that 
\beq\label{23.E.copas}
P_{\textup{as}}A=1+C+F
\eeq
for an $F\in L^\mu (M,{\bf{g}}_l)$ which is flat of order $N-1 $ and a $C\in L^{\infty }(M,{\bf{g}}_l).$
\end{thm}
\begin{rem}
We have a staightforward analogue of Corollary \ref{13.C.as} in the case of a manifold $M$ with edge.
\end{rem}
\nt Some elements of the parametrix construction can be organised in an alternative way when $A$ is an edge-degenerate differential operator. In this case we have locally near the edge $A=r^{-\mu }\Op_y\op_M^{\gamma -n/2}(h)$ for a Mellin symbol $h(r,y,z,\eta )=\sum_{j+|\alpha |\leq \mu }a_{j\alpha }(r,y)z^j(r\eta )^\alpha ,a_{j\alpha }(r,y)\in C^\infty (\overline{\R}_+,\textup{Diff}^{\mu -(j+|\alpha|) }(X)),$ cf. the formula \eqref{iedeg}. Let us consider
\beq\label{23.E.asp1mj4}
A_0=r^{-\mu }\Op_y\op_M^{\gamma -n/2}(h_0)\,\,\,\mbox{for}\,\,\,h_0(r,y,z,\eta )=\sum_{j+|\alpha |\leq \mu }a_{j\alpha }(0,y)z^j(r\eta )^\alpha, 
\eeq
and construct $P_0$ under the assumption of the ellipticity of $A_0.$ The other summands of \eqref{23.E.aspa} can be obtained by the above-mentioned iterative process.
We employ the identity
\beq\label{23.E.amelid}
r^{-\mu }\op_M^{\gamma -n/2}(h_0)r^\mu =\op_M^{\gamma -n/2-\mu }(T^{-\mu }h_0) 
\eeq
and compute a Mellin symbol $m^{(-1)}$ by the ansatz
\beq\label{23.E.ameLe}
\begin{split}
\big(\Op_yr^{-\mu }&\op_M^{\gamma -n/2}(h_0)\big)\big(\Op_yr^\mu \op_M^{\gamma -n/2-\mu }(m^{(-1)})\big)\\
&=\big(\Op_y\op_M^{\gamma -n/2-\mu }(m)\big)\big(\Op_y \op_M^{\gamma -n/2-\mu }(m^{(-1)})\big)\sim 1.
\end{split} 
\eeq
for $m(r,y,z,\eta ):=\tilde{m}(y,z,r\eta ),\tilde{m}(y,z,\tilde{\eta }):=(T^{-\mu }\tilde{h}_0)(y,z,\tilde{\eta })$ where $h_0(r,y,z,\eta )=\tilde{h}_0(y,z,r\eta ).$ The meaning of $\sim$ will be specified below.  \\

\nt The ellipticity of $A_0$ has the consequence that $\tilde{m}(y,z,\tilde{\eta })\in C^\infty (\Omega ,L^\mu _{\textup{cl}}(X;\Gamma _\beta \times \R^q_{\tilde{\eta }}))$ is parameter-dependent elliptic, for every $\beta \in \R.$ Choose a parameter-dependent parametrix $\tilde{f}(y,z,\tilde{\eta })\in C^\infty (\Omega ,L^{-\mu }_{\textup{cl}}(X;\Gamma _\beta \times \R^q_{\tilde{\eta }}))$ for fixed $\beta $ (a simple consideration shows that this can be done including smooth dependence in $y\in \Omega $.) Applying the kernel cut-off operator to $\tilde{f}(y,z,\tilde{\eta })$ gives us a $\tilde{k}_0(y,z,\tilde{\eta })\in C_0^{\infty }(\Omega ,M^{-\mu }_{\mathcal{O}}(X;\R^q_{\tilde{\eta }}))$ such that also $\tilde{k}_0(y,z,\tilde{\eta })$ is a parameter-dependent parametrix for $(z,\tilde{\eta })\in \Gamma _\beta \times \R^q.$ Then, since $\tilde{m}(y,z,\tilde{\eta })\tilde{k}_0(y,z,\tilde{\eta }) $ is equal to $1$ for all $(z,\tilde{\eta })\in \Gamma _\beta \times \R^q$ modulo $C^\infty (\Omega ,L^{-\infty  }(X;\Gamma _\beta \times \R^q_{\tilde{\eta }}))$ it follows also a similar eqivalence for all $(z,\tilde{\eta })\in \Gamma _\delta \times \R^q$ for every real $\delta .$ Let us now find a Mellin symbol as an asymptotic sum $\tilde{m}^{(-1)}(r,y,z,\tilde{\eta })\sim \sum_{\kappa =0}^\infty \tilde{k}_\kappa (r,y,z,\tilde{\eta })$ in $C^\infty (\overline{\R}_+\times \Omega ,M^{-\mu }_{\mathcal{O}}(X;\R^q_{\tilde{\eta }}))$ with $k_\kappa (r,y,z,\newline\eta ) =\tilde{k}_\kappa (r,y,z,r\eta ),$ by successively solving the equations
\beq\label{23.E.aspk}
\sum_{|(\alpha _0,\alpha ')|+\kappa =l}1/\alpha _0!\alpha '!\big(\partial _z^{\alpha _0}\partial _\eta ^{\alpha '}m\big)\big((-r\partial _r)^{\alpha _0}D_y^{\alpha '}k_\kappa \big)=0 \,\,\,\mbox{for all} \,\,\,l=1,2,\dots.
\eeq
$\tilde{k}_0(y,z,\tilde{\eta })$ is already determined. For $l=1$ we obtain
\beq\label{23.E.aspksucc}
k_1=-k_0\sum_{|\alpha _0,\alpha '|=1}\big(\partial _z^{\alpha _0}\partial _\eta ^{\alpha '}m\big)\big((-r\partial _r)^{\alpha _0}D_y^{\alpha '}k_0 \big)
\eeq
which is of the form $k_1(r,y,z,\eta )=\tilde{k}_1 (r,y,z,r\eta )$ for a $\tilde{k}_1 (r,y,z,\tilde{\eta })\in C^\infty (\overline{\R}_+\times \Omega ,M^{-\mu -1}_{\mathcal{O}}(X;\R^q_{\tilde{\eta }})).$ More generally, the solution of \eqref{23.E.aspk} for arbitrary $l\geq 1$ has form $k_l(r,y,z,\eta )=\tilde{k}_l (r,y,z,r\eta )$ for a $\tilde{k}_l (r,y,z,\tilde{\eta })\in C^\infty (\overline{\R}_+\times \Omega ,M^{-\mu -l}_{\mathcal{O}}(X;\R^q_{\tilde{\eta }})).$ In other words the above-mentioned asymptotic sum for $\tilde{k}(r,y,z,\tilde{\eta })$ can be carried out in the desired form. Note that, although in the case of a differential operator this process is not finite, we always have $\partial _z^{\alpha _0}\partial _\eta ^{\alpha '}m=0$ as soon as $|\alpha _0,\alpha '|>\mu ,$ i.e. in \eqref{23.E.aspk} it suffices to take the sum over all $|\alpha _0,\alpha '|\leq \mu .$ Clearly the indicated method of constructing a parametrix in the edge calculus also works in the pseudo-differential case. What we did so far is a new computation of the non-smoothing holomorphic Mellin symbol for the parametrix, here carried out for $A_0.$ As we see, this Mellin symbol is not necessarily constant in $r$ near $0.$ However, in the constructed asymptotic sum for $\tilde{m}^{(-1)}$ we can easily split up the $r$-independent contributions. The remaining part of the parametrix construction is similar as before.

\subsection{Tools from the singular analysis} 
We first recall the definition of the weighted Kegel-spaces $\Kcal^{s,\gamma }(X^\wedge)$ which is most simple in the case of the unit sphere $S^n$ in $\R^{n+1}.$ In Section 1.1
 we defined the weighted spaces $\Hcal^{s,\gamma }(X^\wedge).$ For any fixed cut-off function $\omega $ on the half-axis we set
 $$\Kcal^{s,\gamma }((S^n)^\wedge):=\{\omega u+(1-\omega )v:u\in \Hcal^{s,\gamma }((S^n)^\wedge),v\in H^s(\R^{n+1})\}$$
 for $s,\gamma \in \R.$ The definition is independent of the choice of $\omega .$ The definition for general $X$ refers to a localisation on $U^\wedge=\R_+\times U$ for any coordinate neighbourhood $U$ on $X.$ Let us fix a diffeomorphism
 $$\chi :\R_+\times U\rightarrow \Gamma $$
 for a conical set $\Gamma \subset \R^{n+1}\setminus \{0\}$ such that $\chi (r,x)=r\chi (1,x),r>0,$ for a diffeomorphism $\chi (1,\cdot):U\rightarrow \Gamma \cap S^n.$ Then $\Kcal^{s,\gamma }(X^\wedge)$ is defined to be the set of all $u\in H^s_{\textup{loc}}(X^\wedge)$ such that $\varphi u\circ \chi ^{-1}\in \Kcal^{s,\gamma }((S^n)^\wedge)$ for every such $\chi $ and every $\varphi \in C_0^\infty (U).$\\
 
 \nt From the definition it follows that
 $$\Kcal^{0,0 }(X^\wedge)=\Hcal^{0,0}(X^\wedge)=r^{-n/2}L^2(\R_+\times X).$$
 In $\Kcal^{s,\gamma }(X^\wedge)$ we can easily introduce a scalar product that turns the space to a Hilbert space. In $\Kcal^{s,\gamma }(X^\wedge)$ we consider the group action $\kappa =\{\kappa_\lambda \}_{\lambda \in \R_+} $ defined as
 $$\kappa_\lambda:u(r,x)\rightarrow \lambda ^{n/2}u(\lambda r,x),\,\,\lambda \in \R_+,$$
 cf. the notation in connection with \eqref{21.E.sespac}. Also the spaces $\Kcal^{s,\gamma ;e}(X^\wedge):=\langle r\rangle ^{-e}\Kcal^{s,\gamma }\newline (X^\wedge),s,\gamma ,e\in \R,$ and the Fr\'echet subspaces with asymptotics $\Kcal_P^{s,\gamma ;e}(X^\wedge)$ admit the action of $\kappa .$\\
 
 \nt Let us now recall the definition of operator-valued symbols in the set-up of group actions in the involved spaces. Let $H$ and $\tilde{H}$ be Hilbert spaces with group actions $\kappa $ and $\tilde{\kappa}, $ respectively. Then 
\beq\label{24.E.s}
S^\mu (U\times \R^q;H,\tilde{H})\,\,\,\mbox{for}\,\,\,\mu \in \R\,\,\,\mbox{and open}\,\,\,U\subseteq \R^p
\eeq 
is defined to be the set of all $a(y,\eta )\in C^\infty (U\times \R^q,\mathcal{L}(H,\tilde{H}))$ satisfying the symbolic estimates
\beq\label{24.E.se}
\| \tilde{\kappa }^{-1}_{\langle\eta \rangle}\{D_y^\alpha D_\eta ^\beta a(y,\eta )\}\kappa _{\langle\eta \rangle}\|_{ \mathcal{L}(H,\tilde{H})}\leq c\,\langle\eta \rangle^{\mu -|\beta |}
\eeq 
for all $(y,\eta )\in K\times \R^q$ for any $K\subset \subset U$ and all multi-indices $\alpha \in \N^p,\beta \in \N^q,$ for constants $c=c(\alpha ,\beta ,K)>0.$ The elements of \eqref{24.E.s} are referred to as operator-valued symbols (in the frame of twisted homogeneity). By twisted homogeneity of order $\mu $ of a function $a_{(\mu )}(y,\eta )\in C^\infty (U\times (\R^q\setminus \{0\}),\mathcal{L}(H,\tilde{H}))$ we understand the relation
\beq\label{24.E.ho}
a_{(\mu )}(y,\lambda \eta )=\lambda ^\mu \tilde{\kappa }_\lambda a_{(\mu )}(y,\eta )\kappa _\lambda ^{-1}\,\,\,\mbox{for all}\,\,\,\lambda \in \R_+.
\eeq 
An $a_\mu (y,\eta )\in C^\infty (U\times \R^q,\mathcal{L}(H,\tilde{H}))$ is said to be twisted homogeneous of order $\mu $ for large $|\eta |$ if there is a constant $C>0$ such that the relation $a_\mu (y,\lambda \eta )=\lambda ^\mu \tilde{\kappa }_\lambda a_\mu (y,\eta )\kappa _\lambda ^{-1}$ holds for all $|\eta |\geq C $ and $\lambda \geq 1.$ We often employed the fact that any such $a_\mu (y,\eta )$ belongs to the space \eqref{24.E.s}. By 
\beq\label{24.E.scl}
S_{\textup{cl}}^\mu (U\times \R^q;H,\tilde{H})
\eeq  
we denote the subset of all (classical) elements $a(y,\eta )$ of \eqref{24.E.s} such that there are functions $a_{\mu -j}(y,\eta )\in C^\infty (U\times \R^q,\mathcal{L}(H,\tilde{H})),$ twisted homogeneous of order $\mu -j$ for large $|\eta |,$ such that $a(y,\eta )-\sum_{j=0}^Na_{\mu -j}(y,\eta )\in S^{\mu-(N+1)} (U\times \R^q;H,\tilde{H})$ for every $N\in \N.$ \\

\nt These notions have a straighforward generalisation to the case of Fr\'echet spaces $E$ and/or $\tilde{E}$ with group action $\kappa $ and $\tilde{\kappa},$ respectively. In particular, this is applied in Definition \ref{21.D.seq4}.



\begin{thebibliography}{1}


\bibitem{Atiy8} M.F. Atiyah, V. Patodi, and I.M. Singer,
\textit{Spectral asymmetry and {R}iemannian geometry {I}, {II}, {III}}
Math. Proc. Cambridge Philos. Soc.,
\textbf{77,78,79} (1975, 1976, 1976), 43-69, 405-432, 315-330.

\bibitem{Bout1} L. Boutet de Monvel, \textit{Boundary problems for pseudo-differential operators}, Acta Math. \textbf{126} (1971), 11-51.

\bibitem{Dine4} N. Dines, \textit{Elliptic operators on corner manifolds}, Ph.D. thesis, University of Potsdam, 2006.

\bibitem{Egor1} Ju. V. Egorov and B.-W. Schulze,  \textit{Pseudo-differential operators, singularities, applications}, Oper. Theory: Adv. Appl. \textbf{93}, Birkh{\"a}user Verlag, Basel, 1997.

\bibitem{Eski2} G.I. Eskin, \textit{Boundary value problems for elliptic pseudodifferential equations}, Transl. of Nauka, Moskva, 1973, Math. Monographs, Amer. Math. Soc. \textbf{52}, Providence, Rhode Island 1980.

\bibitem{Flad3} H.-J. Flad, G. Harutyunyan, R. Schneider, and B.-W. Schulze \textit{Explicit Green operators for quantum mechanical Hamiltonians.I. The hydrogen atom}, arXiv:1003.3150v1 [math.AP], 2010. Manuscripta Mathematica (to appear).

\bibitem{Flad4} H.-J. Flad, G. Harutyunyan, \textit{Ellipticity of quantum mechanical Hamiltonians in the edge algebra}, Proc. AIMS Conference on Dynamical Systems, Differential Equations and Applications, Dresden, 2010.

\bibitem{Gil2} J.B. Gil, B.-W. Schulze,  and J. Seiler, \textit{Cone pseudodifferential operators in the edge symbolic calculus}, Osaka J. Math. \textbf{37} (2000), 221-260.

\bibitem{Gohb3} I.C. Gohberg and E.I. Sigal, \textit{An operator generalization of the logarithmic residue theorem and the theorem of {R}ouch{\'e}}, Math. USSR Sbornik \textbf{13}, 4 (1971), 603-625.

\bibitem{Haru13} G. Harutjunjan and B.-W. Schulze, \textit{Elliptic mixed, transmission and singular crack problems}, European Mathematical Soc., Z\" urich, 2008.

\bibitem{Kapa10} D. Kapanadze  and B.-W. Schulze, \textit{Crack theory and edge singularities}, Kluwer Academic Publ., Dordrecht, 2003.

\bibitem{Kond1} V.A. Kondratyev, \textit{Boundary value problems for elliptic equations in domains with conical points}, Trudy Mosk. Mat. Obshch. \textbf{16}, (1967), 209-292.

\bibitem{Kuma1} H. Kumano-go, \textit{Pseudo-differential operators}, The MIT Press, Cambridge, Massachusetts and London, England, 1981.


\bibitem{Remp2} S. Rempel and B.-W. Schulze, \textit{Index theory of elliptic boundary problems}, Akademie-Verlag, Berlin, 1982.

\bibitem{Remp1} S. Rempel and B.-W. Schulze, \textit{Parametrices and boundary symbolic calculus for elliptic boundary problems without transmission property}, Math. Nachr. \textbf{105}, (1982), 45-149.

\bibitem{Remp7} S. Rempel and B.-W. Schulze, \textit{Complete {M}ellin and {G}reen symbolic calculus in spaces with conormal asymptotics}, Ann. Glob. Anal. Geom. \textbf{4}, 2 (1986), 137-224.

\bibitem{Schu28} B.-W. Schulze, \textit{Regularity with continuous and branching asymptotics for elliptic operators on manifolds with edges}, Integral Equations Operator Theory \textbf{11} (1988), 557-602.

\bibitem{Schu32} B.-W. Schulze, \textit{Pseudo-differential operators on manifolds with edges}, Teubner-Texte zur Mathematik, Symp. ``Partial Differential Equations'', Holzhau 1988, \textbf{112}, Leipzig, 1989, pp. 259-287.

\bibitem{Schu33} B.-W. Schulze, \textit{{M}ellin representations of pseudo-differential operators on manifolds with corners}, Ann. Glob. Anal. Geom. \textbf{8}, 3 (1990), 261-297.

\bibitem{Schu2} B.-W. Schulze, \textit{Pseudo-differential operators on manifolds with singularities}, North-Holland, Amsterdam, 1991.

\bibitem{Schu20} B.-W. Schulze, \textit{Boundary value problems and singular pseudo-differential operators}, J. Wiley, Chichester, 1998.

\bibitem{Schu37} B.-W. Schulze, \textit{An algebra of boundary value problems not requiring {Shapiro-Lopatinskij} conditions}, J. Funct. Anal. \textbf{179} (2001), 374-408.

\bibitem{Schu42} B.-W. Schulze and J. Seiler, \textit{Edge operators with conditions of {T}oeplitz type}, J. of the Inst. Math. Jussieu, \textbf{5}, 1 (2006), 101-123.

\bibitem{Schu57} B.-W. Schulze, \textit{The iterative structure of corner operators}, arXiv: 0901.1967v1 [math.AP], 2009.

\bibitem{Seil2} J. Seiler, \textit{The cone algebra and a kernel characterization of {Green} operators}, Oper. Theory Adv. Appl. \textbf{125} Adv. in Partial Differential Equations ``Approaches
to Singular Analysis'' (J. Gil, D. Grieser, and M. Lesch, eds.), Birkh\"auser, Basel, 2001, pp. 1-29.

\bibitem{Witt3} I. Witt, \textit{On the factorization of meromorphic Mellin symbols}, Advances in Partial Differential Equations (Approaches to Singular Analysis) (S. Albeverio, M. Demuth, E. Schrohe,  and B.-W. Schulze, eds.), Oper. Theory: Adv. Appl., Birkh\" auser Verlag, Basel, 2002, pp. 279-306.









\end{thebibliography}
\end{document}